\documentclass[reqno,11pt]{amsart}

\usepackage{amsmath,amssymb,amsthm,tikz,mathtools}
\usepackage{mathrsfs}       % Script fonts
\usepackage{hyperref}
\usepackage{xcolor}         % Colored text
\usepackage{esint}
\usepackage{tikz}           % Drawing
\usepackage{verbatim}       % Commenting
\usepackage{setspace}       % Line spacing
\usepackage[normalem]{ulem}
\usepackage{cancel}
\title[Fractional $p-$caloric functions are Lipschitz]{Fractional $p-$caloric functions are Lipschitz}

\author[D. Jesus]{David Jesus}
\address{Applied Mathematics and Computational Sciences (AMCS), Computer, Electrical and Mathematical Sciences and Engineering Division (CEMSE), King Abdullah University of Science and Technology (KAUST), Thuwal, 23955-6900, Kingdom of Saudi Arabia}
\email{david.dejesus@kaust.edu.sa}

\author[A. Sobral]{Aelson Sobral}
\address{Applied Mathematics and Computational Sciences (AMCS), Computer, Electrical and Mathematical Sciences and Engineering Division (CEMSE), King Abdullah University of Science and Technology (KAUST), Thuwal, 23955-6900, Kingdom of Saudi Arabia}
\email{aelson.sobral@kaust.edu.sa}

\author[J.M. Urbano]{Jos\'{e} Miguel Urbano\textsuperscript{\dag}}
\thanks{\textsuperscript{\dag}Corresponding author.}
\address{Applied Mathematics and Computational Sciences (AMCS), Computer, Electrical and Mathematical Sciences and Engineering Division (CEMSE), King Abdullah University of Science and Technology (KAUST), Thuwal, 23955-6900, Kingdom of Saudi Arabia and CMUC, Department of Mathematics, University of Coimbra, 3000-143 Coimbra, Portugal}
\email{miguel.urbano@kaust.edu.sa}

\newtheorem{theorem}{Theorem}[section]
\newtheorem{lemma}{Lemma}[section]

\newtheorem{proposition}{Proposition}[section]
\newtheorem{definition}{Definition}[section]

\newcommand{\Tail}{\operatorname{Tail}}

\newcommand{\abs}[1]{\left|#1\right|}

\newcommand{\R}{\mathbb{R}}

\newcommand{\p}{\partial}
\newcommand{\LL}{\mathcal{L}}
\newcommand{\II}{\mathcal{I}}

\newcommand{\dd}{\, \mathrm{d}}

\newcommand{\C}{\mathcal{C}}

\newcommand{\intav}[1]{\mathchoice 
  {\mathop{\vrule width 6pt height 3 pt depth -2.5pt \kern -8pt \intop}\nolimits_{\kern -6pt#1}} 
  {\mathop{\vrule width 5pt height 3 pt depth -2.6pt \kern -6pt \intop}\nolimits_{#1}}
  {\mathop{\vrule width 5pt height 3 pt depth -2.6pt \kern -6pt \intop}\nolimits_{#1}}
  {\mathop{\vrule width 5pt height 3 pt depth -2.6pt \kern -6pt \intop}\nolimits_{#1}}}

\numberwithin{equation}{section}
\begin{document}


% 35B65 Smoothness and regularity of solutions to PDEs

% 35R11 Fractional partial differential equations

% 35B51 Comparison principles in context of PDEs

% 35D40 Viscosity solutions to PDEs

% 35K92 Quasilinear parabolic equations with $p-$Laplacian


\begin{abstract} 
We study the parabolic fractional $p-$Laplace equation 
$$
    \p_t u+(-\Delta_p)^su = 0
$$
in the degenerate range \(2 < p < 2/(1-s)\). We show that weak solutions are Lipschitz continuous in space and, if \(p > 1/(1-s)\), also in time. We also prove a comparison principle for both weak and viscosity solutions, and establish the equivalence between the two notions of solution.
\end{abstract}  

\date{\today}

\maketitle

\section{Introduction} \label{s:intro}

In this paper, we study solutions to the parabolic fractional $p-$Laplace equation
\begin{equation}\label{eq:parabolic-fractional-p-laplacian}
\p_t u+(-\Delta_p)^su = 0,
\end{equation}
for $s\in(0,1)$ and \(2 < p < 2/(1-s)\), where
\[
  (-\Delta_p)^su(x,t) \coloneqq \mathrm{P.V.} \int_{\R^d} \frac{|u(x,t)-u(y,t)|^{p-2}}{|x-y|^{d+sp}}(u(x,t) - u(y,t))\,dy.
\]
The interest in this problem stems not only from its rich mathematical structure, but also from its relevance in applications, particularly in models describing the rate at which individuals arrive at a given position \(x\) from all other locations, when both nonlinearity and long-range interactions are important. Moreover, equation \eqref{eq:parabolic-fractional-p-laplacian} arises naturally as the gradient flow associated with the Gagliardo \(W^{s,p}-\)seminorm, defined for \(s \in (0,1)\) and \(p \in [1,\infty)\) by
\begin{equation*}
    [u]_{W^{s,p}(\R^d)} \coloneqq \left( \iint_{\R^d \times \R^d} \frac{|u(x)-u(y)|^p}{|x-y|^{d+sp}} \, dx \, dy \right)^{1/p}.
\end{equation*}

A substantial regularity theory has been developed for nonlocal equations driven by the fractional \(p-\)Laplacian,  an operator first introduced in \cite{AMRT, IN}. In the stationary setting, both weak and viscosity solutions are known to satisfy a Harnack inequality and to enjoy H\"older continuity and higher integrability properties; see \cite{BDLMS, BDLM, BL, BLS1, DCKP}. More recently, Lipschitz estimates were obtained in \cite{BS,BT}, and \(C^{1,\alpha}-\)regularity was established in \cite{GJS}. By contrast, in the evolutionary setting, the regularity theory remains fairly limited, with the currently available results concerning only the H\"older regularity of weak solutions; see \cite{BLS2, GLT, L}.  

The main contribution of this work is the following regularity result stated in the parabolic cylinders $Q_r \coloneqq B_r(0)\times (-r^2,0]$. For the definition of the tail space $L^{p-1}_{sp}(\R^d)$, see subsection \ref{ss:constnotunv}.

\begin{theorem}\label{t:Lipschitz}
Let $p > 2$ and $s \in (0,1)$ be such that
\[
    q_c \coloneqq - 1 + p(1-s) < 1.
\]
Let \(u\) be a weak solution to \eqref{eq:parabolic-fractional-p-laplacian} in \(Q_1\), and assume that
\begin{equation}\label{eq:LiDi_fora}
    u \in C_{\mathrm{loc}} (-1,0;L^{p-1}_{sp}(\R^d)).
\end{equation}
Set
\[
    \mathcal{M} \coloneqq \|u\|_{L^\infty(Q_{3/4})} + \sup_{t \in (-(3/4)^2,0]} \|u(\cdot,t)\|_{L^{p-1}_{sp}(\R^d)}.
\]
Then, the following assertions hold:

\begin{enumerate}

\item If \(q_c \leq 0\), then for every $\alpha < \frac{1}{1-q_c}$, there exists a constant \(C>0\), depending only on \(d\), \(p\), \(s\), and \(\alpha\), such that
\[
    |u(x,t)-u(y,\tau)| \leq C \mathcal{M} \Bigl(|x-y| + \mathcal{M}^{(p-2)\alpha}|t-\tau|^\alpha\Bigr)
\]
for all \((x,t),(y,\tau) \in Q_{1/2}\).

\medskip

\item If \(q_c>0\), then there exists a constant \(C>0\), depending only on \(d\), \(p\), and \(s\), such that
\[
    |u(x,t)-u(y,\tau)| \leq C \mathcal{M} \Bigl(|x-y| + \mathcal{M}^{p-2}|t-\tau|\Bigr)
\]
for all \((x,t),(y,\tau) \in Q_{1/2}\).
\end{enumerate}
\end{theorem}

The constant $q_c$ appearing in the theorem plays an important role in our analysis. In particular, we prove Lipschitz regularity, and the condition $q_c>0$ is precisely the one that ensures that the local contribution to $(-\Delta_p)^s |x|$ is finite.

Let us briefly highlight the main novelties and difficulties. Our result brings two main improvements to the existing parabolic literature. First, it establishes spatial Lipschitz regularity across the full degenerate range \(2 < p < 2/(1-s)\). In the elliptic case, the proof of the \(C^{1,\alpha}-\)regularity in \cite{GJS} relies crucially on these Lipschitz estimates, and here we carry out the corresponding step in the parabolic setting. Second, in the range \( 0 < q_c < 1 \), we prove that solutions are also Lipschitz continuous in time, which is optimal in view of the counterexample given in \cite[Remark 1.3]{BLS2}. For the case \(q_c \leq 0\), we obtain H\"older regularity in time for the same exponents as in \cite{BLS2}, but our proof consists of a simple barrier argument.

The proof of Theorem \ref{t:Lipschitz} proceeds in two steps. We first establish the spatial Lipschitz estimate by adapting the Ishii--Lions doubling-of-variables method, introduced in \cite{IL} for local equations, to the parabolic nonlocal setting. The argument starts with a spatially doubled functional at a fixed final time \(t_0\), but in order to obtain legitimate viscosity inequalities, one must also double the time variable. More precisely, after considering
\[
    u(x,t)-u(y,t)-L\omega(|x-y|)-L_2\psi(x)-L_2(t_0-t)^{1+\beta_*},
\]
we double the time variable and introduce the additional penalization \(K(t-\tau)^2\). At a maximum point \((x_K,y_K,t_K,\tau_K)\), the viscosity inequalities are then applied to two patched test functions touching \(u\) from above at \((x_K,t_K)\) and from below at \((y_K,\tau_K)\). Subtracting these inequalities reduces the estimate to a lower bound for the difference of two fractional \(p\)-Laplacians evaluated at two nearby, but not identical, times.

This is the main point where the parabolic argument differs from the elliptic one. In the elliptic setting, the usual spatial doubling immediately gives the required comparison of increments. Here, the time-doubling procedure produces two time levels \(t_K\) and \(\tau_K\). The interior part can be controlled using the already available H\"older continuity in space-time, while the exterior part involves the tail of \(u(\cdot,t_K)\) and \(u(\cdot,\tau_K)\), where no regularity is available. We overcome this by first fixing the spatial Lipschitz parameter \(L\) and only afterward choosing \(K\) sufficiently large so that \(t_K-\tau_K\) is small, thereby controlling the tail terms. This use of the continuity of \(u\) with values in \(L^{p-1}_{sp}(\mathbb R^d)\) is essential for closing the parabolic Ishii--Lions estimate; however, since \(L\) is chosen before \(K\), the final Lipschitz constant remains independent of the modulus of continuity of the tail. Once this estimate is established for a given modulus \(\omega\), the elliptic bootstrapping mechanism can be recovered, leading first to improved H\"older exponents and ultimately to the spatial Lipschitz bound.

To obtain regularity in the time variable, we exploit the spatial Lipschitz regularity to construct suitable barriers, following a fairly standard strategy in the local setting. In our case, however, the nonlocal part of the equation allows us to use barrier functions with substantially lower spatial regularity, which, in turn, enables us to transfer additional regularity to the time variable. This yields the same regularity regime as in \cite{BLS2} for $q_c \leq 0$. Somewhat surprisingly, for \( 0 < q_c < 1 \), the argument also gives Lipschitz regularity in time, since in this range, the kernel is only mildly singular. One might even hope to derive higher regularity estimates by regarding \(\partial_t u \in L^\infty\) as a bounded right-hand side, but this approach fails since solutions to the corresponding inhomogeneous equation are only expected to be of class \(C^{1,\alpha}\) in precisely the disjoint range \(q_c<0\).

A subtle but important point is that both arguments leading to the space-time estimates are based entirely on viscosity methods. Consequently, one needs a \textit{weak-implies-viscosity} result, which is precisely why the assumption
\[
    u \in C_{\mathrm{loc}}(-1,0;L^{p-1}_{sp}(\R^d))
\]
is imposed in \eqref{eq:LiDi_fora}, ensuring that the tail quantity of \(u\)  depends continuously on time. This requirement is quite natural: without it, the nonlocal part of \((-\Delta_p)^s u(x,\cdot)\), which should be regarded as part of the data, would not be continuous in time. For example, if a test function touches $u$ at a point $(x_0,t_0)$, we need to guarantee that the tail of $u$ at $t_0$ is well defined. We could weaken this assumption in the definition of viscosity solution to $\|u(\cdot,t)\|_{L^{p-1}_{sp}(\R^d)}$ being well-defined and finite for every time. However, we use in a crucial way that the tail is continuous in time in many instances of this paper. Nevertheless, the regularity estimates established in Theorem \ref{t:Lipschitz} do not depend on the modulus of continuity of the tail. A similar condition already appears in \cite{CLD}, for $p=2$.

In fact, our results go beyond showing that weak solutions are viscosity solutions. We also establish several key results, including a comparison principle for both notions of solution, which can be regarded as the parabolic counterparts of those obtained in \cite{KKL} and \cite{KKP}. We further discuss the converse implication, namely, whether viscosity solutions are also weak solutions, thereby providing a fairly comprehensive understanding of the viscosity framework. The overall strategy is inspired by the ideas developed in \cite{KKP}, as well as by \cite{JJ} in the local case. We obtain this implication assuming the solution at the tail is continuous and bounded, which is stronger than the assumptions for the other implication. However, these assumptions are natural and were also considered in \cite{CS, HL}. The implication with full generality would follow from a parabolic obstacle argument as in \cite[Lemma 9]{KKP}. However, we would need the parabolic counterpart of \cite{KKP0}.

To streamline the presentation of the main ideas, we restrict ourselves to the homogeneous setting. However, with no substantial additional difficulty, it should be possible to obtain spatial Lipschitz regularity also in the presence of a bounded source term, at least in the range \(2 < p < 1/(1-s)\). The barrier argument would likewise extend seamlessly whenever the right-hand side is bounded. We also believe that the same strategy could be applied to solutions of the doubly nonlinear equation
\[
    |\partial_t u|^{p-2}\partial_t u + (-\Delta_p)^s u = 0,
\]
see \cite{HL} and the references therein for further details. 

The paper is organized as follows. We dedicate Section \ref{s:prelim} to the mathematical setup and to recalling known results. In Section \ref{s:properties-solutions}, we prove the comparison principle for both weak and viscosity solutions, as well as the equivalence between weak and viscosity solutions. In Section \ref{s:lipschitz-regularity}, we prove the Lipschitz regularity in space. The time regularity is proved in Section \ref{s:time-variable-estimates}, which also contains the proof of Theorem \ref{t:Lipschitz}.

\section{Mathematical setup and preliminaries}\label{s:prelim}

In this section, we collect some preliminary results needed for the remainder of the paper. 

\subsection{Notation}\label{ss:constnotunv}

Throughout the paper, $\Omega \subset \R^d$ denotes a bounded smooth domain. We define the parabolic cylinders
\[
    Q_r(x_0,t_0) \coloneqq B_r(x_0)\times (t_0-r^2,t_0] \quad \text{and} \quad Q_r \coloneqq B_r(0)\times (-r^2,0].
\]

We define $C^{2}(\Omega\times(t_1,t_2])$ as the space of functions $v \colon \Omega\times(t_1,t_2] \to\mathbb R$ that are $C^{2}$ in the spatial variables and $C^{1}$ in time. Similarly, we define $C^{1,1}(\Omega\times(t_1,t_2])$ as the space of functions that are $C^{1,1}$ in the spatial variables and $C^{0,1}$ in time. We also define the tail spaces $L^{p-1}_{sp}(\R^d)$ consisting of all functions $w\colon \R^d \to \R$ satisfying
\[
    \|w\|_{L^{p-1}_{sp}(\R^d)} \coloneqq \left(\int_{\R^d} \frac{|w(x)|^{p-1}}{1 + |x|^{d + sp}}\dd x\right)^{\frac{1}{p-1}} < \infty.
\]

We use the notation $\gamma^-$ to denote an arbitrary number strictly smaller than $\gamma$. Throughout the paper, we say that a constant is universal if it depends only on $d$, $p$, and $s$. We usually denote small universal constants by $c$ and large ones by $C$, which may vary from line to line. We may enumerate them when necessary.
 
\subsection{Weak/viscosity solutions}

We introduce the notions of weak and viscosity solution to equation \eqref{eq:parabolic-fractional-p-laplacian}.

\begin{definition}\label{def:weak-sol-definition}
We say that $u\colon \mathbb{R}^d \times (t_1,t_2] \to \mathbb{R}$ is a weak subsolution to \eqref{eq:parabolic-fractional-p-laplacian} in $\Omega \times (t_1,t_2]$, if 
\[
    u \in C_{\mathrm{loc}}\left(t_1,t_2;L^2_{\mathrm{loc}}(\Omega)\right) \cap L^p_{\mathrm{loc}}\left(t_1,t_2;W^{s,p}_{\mathrm{loc}}(\Omega) \right)\cap L_{\mathrm{loc}}^\infty(t_1,t_2;L_{sp}^{p-1}(\R^d))
\]
and for any compact set $\mathcal{K} \subset \Omega$ and any sub-interval $[\tau_1,\tau_2] \subset (t_1,t_2]$ 
\[
\begin{aligned}
&\int_{\mathcal{K}} u(x,t)\varphi(x,t)\,\dd x \bigg|_{\tau_1}^{\tau_2}
 - \int_{\tau_1}^{\tau_2}\int_{\mathcal{K}} u(x,t)\,\partial_t\varphi(x,t)\,\dd x \dd t \\
&\qquad
 + \int_{\tau_1}^{\tau_2} \mathcal{E}(u(\cdot,t),\varphi(\cdot,t))\,\dd t
\le 0,
\end{aligned}
\]
for any nonnegative $\varphi$, such that
\[
   \varphi \in L^p(\tau_1,\tau_2;W^{s,p}_0(\mathcal{K})) \cap H^1(\tau_1,\tau_2;L^2(\mathcal{K})), 
\]
where the quantity $\mathcal{E}(u(\cdot,t),\varphi(\cdot,t))$ is defined as
\[
    \iint_{\R^d \times \R^d} \frac{|u(x,t) - u(y,t)|^{p-2}(u(x,t) - u(y,t))(\varphi(x,t) - \varphi(y,t))}{|x-y|^{d+sp}}\dd y\dd x.
\]
A weak supersolution is defined in a similar way. Furthermore, a weak solution is both a weak subsolution and a supersolution.
\end{definition}

Now we define the notion of viscosity solution.

\begin{definition}\label{def:viscosity-solution}
We say that $u$ is a viscosity subsolution to \eqref{eq:parabolic-fractional-p-laplacian} in $\Omega \times (t_1,t_2]$ if 
\[
    u\in \mathrm{USC}(\Omega\times(t_1,t_2])\cap C_{\mathrm{loc}}(t_1,t_2;L_{sp}^{p-1}(\R^d))
\]
and whenever there is a function $\phi\in C^2(Q_r(x_0,t_0))$ for some $Q_r(x_0,t_0) \subset \Omega \times (t_1,t_2]$, with $\phi$ touching $u$ from above at $(x_0,t_0)$ in $Q_r(x_0,t_0)$, we have 
\[
    \partial_t \phi(x_0,t_0) + (-\Delta_p)^s \phi_r (x_0,t_0)\leq 0,
\] 
where
\[
\phi_r(x,t)=\left\{\begin{array}{ll}
\phi(x,t) & \text{for}\; (x,t)\in Q_r(x_0,t_0),
\\[2mm]
u(x,t) & \text{otherwise}.
\end{array}
\right.
\]
A viscosity supersolution is defined in a similar way. Furthermore, a viscosity solution is both a viscosity subsolution and a supersolution.
\end{definition}

\subsection{Known results and useful estimates}

We begin by recalling the H\"older regularity result from \cite[Theorem 1.1]{L} for weak solutions of \eqref{eq:parabolic-fractional-p-laplacian}, adapted to our setting. 

\begin{theorem}\label{t:NL_Bra_nobo}
Let $u$ be a locally bounded weak solution to 
\[
    \p_t u + (-\Delta_p)^s u = 0
    \qquad \text{in } Q_1.
\]
Set
\[
    \mathcal{M} \coloneqq \|u\|_{L^\infty(Q_{3/4})} + \sup_{t \in (-(3/4)^2,0]} \|u(\cdot,t)\|_{L^{p-1}_{sp}(\R^d)}.
\]
Then there exists a small $\kappa_0\in(0,1)$ such that $u$ is locally $\kappa_0-$H\"older continuous and
\[
    |u(x,t)-u(y,\tau)|\leq C\mathcal{M}\left(|x-y|^{\kappa_0}+\mathcal{M}^{p-2}|t-\tau|^{\kappa_0}\right), 
\]
for $(x,t), (y,\tau) \in Q_{1/2}$.
\end{theorem}

Now, we state a useful estimate for the function $J_p \colon \R \to \R$ defined as $J_p(\tau) = |\tau|^{p-2}\tau$.

\begin{lemma}\label{l:Jp_est}
Suppose $p>1$ and $a, b\in \R$. Then 
\begin{align*}
&J_p(a) -J_p(b)= (p-1) \int_{0}^{1} |b+\tau(a-b)|^{p-2} (a-b)  \dd  \tau, \\
\intertext{ and}
&\frac{1}{C_p} (|b|+|a-b|)^{p-2}\leq \int_0^1 |b+\tau (a-b)|^{p-2} \dd \tau\leq C_p (|b|+|a-b|)^{p-2},
\end{align*}
for some constant $C_p$, depending only on $p$. 
\end{lemma}

Finally, we recall the following result from \cite[Lemma 3.6]{KKL}.

\begin{lemma}\label{l:Korvenpaa-Kuusi-Lindgren}
Let $B_\varepsilon(x) \subset D \Subset \Omega$, $p\geq 2$ and $w \in C^2(D)$. Then
\[
    \left| \mathrm{P.V.} \int_{B_\varepsilon(x)} 
    \frac{J_p(w(x)-w(y))}{|x-y|^{d+sp}} \dd y \right|
    \le c_\varepsilon,
\]
where $c_\varepsilon$ is independent of $x$ and $c_\varepsilon \to 0$ as $\varepsilon \to 0$.
\end{lemma}

\section{Properties of weak and viscosity solutions}\label{s:properties-solutions}

In this section, we establish comparison principles for both weak and viscosity solutions and use them to prove the equivalence between these two notions of solution.

\subsection{Comparison principles}

We start with the comparison principle for weak solutions.

\begin{theorem}\label{t:comparison-weak}
Let $u$ be a weak supersolution and $v$ be weak subsolution of \eqref{eq:parabolic-fractional-p-laplacian} in $\Omega \times (t_1,t_2]$. Let $\mathcal{K}\Subset \Omega$ and $[\tau_1,\tau_2]\subset (t_1,t_2]$. Assume that
\[
    u \ge v, \quad \text{a.e. in } \bigl(\R^d\setminus \mathcal{K} \bigr) \times (\tau_1,\tau_2] \text{ and on } \mathcal{K} \times \{\tau_1\}.
\]
Then
\[
    u \ge v, \quad \text{a.e. in } \mathcal{K} \times (\tau_1,\tau_2].
\]
\end{theorem}

\begin{proof}
Let $\varphi(x,t) \coloneqq \max\{v(x,t)-u(x,t),0\}$. Let $\bar t \in (\tau_1,\tau_2)$. Since $u \geq v$ in $\bigl(\R^d\setminus \mathcal{K} \bigr) \times (\tau_1,\tau_2]$, we have $\varphi \in L^p(\tau_1,\bar t;W^{s,p}_0(\mathcal{K}))$.  Showing that $\varphi \in H^1(\tau_1,\bar t;L^2(\mathcal{K}))$ is more delicate, but a rigorous justification of this step can be obtained using Steklov averages; see \cite[Appendix B]{L}. Testing the equations for $u$ and $v$ with $\varphi$, we obtain
\[
\begin{aligned}
&\int_{\mathcal{K}} u(x,t)\varphi(x,t)\,\dd x \bigg|_{\tau_1}^{\bar t}
 - \int_{\tau_1}^{\bar t}\int_{\mathcal{K}} u(x,t)\,\partial_t\varphi(x,t)\,\dd x \dd t \\
&\qquad
 + \int_{\tau_1}^{\bar t} \mathcal{E}(u(\cdot,t),\varphi(\cdot,t))\,\dd t
\ge 0.
\end{aligned}
\]
and
\[
\begin{aligned}
&\int_{\mathcal{K}} v(x,t)\varphi(x,t)\,\dd x \bigg|_{\tau_1}^{\bar t}
 - \int_{\tau_1}^{\bar t}\int_{\mathcal{K}} v(x,t)\,\partial_t\varphi(x,t)\,\dd x \dd t \\
&\qquad
 + \int_{\tau_1}^{\bar t} \mathcal{E}(v(\cdot,t),\varphi(\cdot,t))\,\dd t
\le 0,
\end{aligned}
\]
where the quantity $\mathcal{E}(\cdot,\cdot)$ is as in Definition \ref{def:weak-sol-definition}.

We first show that, for a.e. $t \in [\tau_1,\bar t]$, we have
\begin{equation}\label{ineq:energy-nonpositive}
    \mathcal{E}(u(\cdot,t),\varphi(\cdot,t)) - \mathcal{E}(v(\cdot,t),\varphi(\cdot,t)) \leq 0.
\end{equation}
Define
\[
    U_t^+ \coloneqq \{x \in \R^d \colon v(x,t) > u(x,t)\} \quad \text{and} \quad U_t^- \coloneqq \{x \in \R^d \colon v(x,t) \leq u(x,t)\},
\]
and notice that $\varphi = v-u$ in $U_t^+$, and $\varphi = 0$ in $U_t^-$. We can then write the LHS of \eqref{ineq:energy-nonpositive} as
\begin{align*}
& \int_{U_t^+}\int_{U_t^+}
\bigl(J_p(u(x,t)-u(y,t)) - J_p(v(x,t)-v(y,t))\bigr)\bigl(\varphi(x,t)-\varphi(y,t)\bigr)\,d\mu\\
&+ 2\int_{U_t^-}\int_{U_t^+}
\bigl(J_p(u(x,t)-u(y,t)) - J_p(v(x,t)-v(y,t))\bigr)\varphi(x,t)\,d\mu
\end{align*}
where $d\mu = |x-y|^{-d-sp}\dd x\dd y$. Each of the terms above is nonpositive. To see this, recall first that, for every $a,b\in\R$, we have
\[
    (J_p(a) - J_p(b))(a-b) \geq 0.
\]
Therefore, for $x,y \in U_t^+$, we have
\[
    \int_{U_t^+}\int_{U_t^+}
\bigl(J_p(u(x,t)-u(y,t)) - J_p(v(x,t)-v(y,t))\bigr)\bigl(\varphi(x,t)-\varphi(y,t)\bigr)\,d\mu \leq 0.
\]
Now, if $x \in U_t^+$ and $y \in U_t^-$, $J_p(u(x,t) - u(y,t)) \leq J_p(v(x,t) - v(y,t))$, which gives
\[  
    \int_{U_t^-}\int_{U_t^+}
\bigl(J_p(u(x,t)-u(y,t)) - J_p(v(x,t)-v(y,t))\bigr)\varphi(x,t)\,d\mu \leq 0.
\]
and thus \eqref{ineq:energy-nonpositive} holds. Subtracting the super and subsolution inequalities and using \eqref{ineq:energy-nonpositive}, we obtain
\begin{align*}
    0 & \leq \int_{\mathcal{K}}u(x,t)\varphi(x,t)\dd x \bigg|_{\tau_1}^{\bar t} - \int_{\tau_1}^{\bar t}\int_{\mathcal{K}}u(x,t) \partial_t \varphi(x,t)\dd x\dd t\\
    & \quad - \int_{\mathcal{K}}v(x,t)\varphi(x,t)\dd x \bigg|_{\tau_1}^{\bar t}  + \int_{\tau_1}^{\bar t}\int_{\mathcal{K}}v(x,t) \partial_t \varphi(x,t)\dd x\dd t\\
    & = \int_{\mathcal{K}}(u(x,t) - v(x,t))\varphi(x,t)\dd x \bigg|_{\tau_1}^{\bar t} + \int_{\tau_1}^{\bar t}\int_{\mathcal{K}}(v(x,t)-u(x,t)) \partial_t \varphi(x,t)\dd x\dd t\\
    & = -\int_{\mathcal{K}}\varphi^2(x,t)\dd x \bigg|_{\tau_1}^{\bar t} + \frac{1}{2}\int_{\tau_1}^{\bar t}\int_{\mathcal{K}}\partial_t(\varphi^2(x,t))\dd x\dd t,
\end{align*}
which gives
\[
    \|\varphi(\cdot,\bar t)\|_{L^2(\mathcal{K})} \leq \|\varphi(\cdot,\tau_1)\|_{L^2(\mathcal{K})}.
\]
Since at $\tau_1$ we have $\varphi(x,\tau_1) = 0$ for all $x \in \mathcal{K}$, we get $\varphi(\cdot,\bar t) \equiv 0$ in $\mathcal{K}$. Since $\bar t$ is arbitrary, this implies $\varphi \equiv 0$ in $\mathcal{K} \times (\tau_1,\tau_2]$, from which follows $u \geq v$ a.e. in $\mathcal{K} \times (\tau_1,\tau_2]$.
\end{proof}

Next, we obtain the comparison principle for viscosity solutions. First, we prove a lemma stating that if we can touch the solution with a sufficiently smooth test function, then $(-\Delta_p)^s u$ is well-defined at the touching point and can be used in the equation satisfied by the test function.

\begin{lemma}\label{l:pointwise-equation}
Let \(u\) be a viscosity subsolution (respectively, supersolution) of \eqref{eq:parabolic-fractional-p-laplacian} in \(\Omega \times (t_1,t_2]\). Assume that \(\varphi \in C^2(Q_r(x_0,t_0))\), where
\[
    Q_r(x_0,t_0) \subset \Omega \times (t_1,t_2],
\]
and that $\varphi$ touches \(u\) from above (respectively, below) at \((x_0,t_0)\) in \(Q_r(x_0,t_0)\). Then
\[
    \partial_t \varphi(x_0,t_0) + (-\Delta_p)^s u(x_0,t_0) \le (\ge)\,0.
\]
\end{lemma}

\begin{proof}
We prove the result for subsolutions since the other case is identical. For each \(0<\rho<r\), consider
\[
    \varphi_\rho(x,t) =
    \begin{cases}
    \varphi(x,t) & \text{if } (x,t)\in Q_\rho(x_0,t_0),\\[2mm]
    u(x,t) & \text{otherwise}.
    \end{cases}
\]
Since \(u\) is a viscosity subsolution of \eqref{eq:parabolic-fractional-p-laplacian}, we have
\[
    \partial_t\varphi(x_0,t_0)+(-\Delta_p)^s\varphi_\rho(x_0,t_0)\le 0.
\]
We decompose \( (-\Delta_p)^s\varphi_\rho(x_0,t_0)\) as the $\mathrm{P.V.}$ of the integral over $B_\rho(x_0)$ and outside. By Lemma \ref{l:Korvenpaa-Kuusi-Lindgren}, the integral over $B_\rho(x_0)$ goes to zero as $\rho\to 0$. Therefore, in the limit, we obtain
\[
    \partial_t\varphi(x_0,t_0)+\mathrm{P.V.}\int_{\R^d}\frac{J_p\!\big(u(x_0,t_0)-u(y,t_0)\big)}{|\,x_0-y\,|^{d+sp}}\dd y\le 0.
\]
This completes the proof. 
\end{proof}

We next establish the comparison principle for viscosity supersolutions and subsolutions to \eqref{eq:parabolic-fractional-p-laplacian}. 

\begin{theorem}\label{t:comparison-viscosity}
Let \(u\) and \(v\) be a viscosity supersolution and a viscosity subsolution, respectively, of \eqref{eq:parabolic-fractional-p-laplacian} in \(\Omega \times (t_1,t_2]\). Assume, in addition, that 
\[
 -u,v \in \mathrm{USC} (\overline{\Omega}\times[t_1,t_2]).
\]
If
\begin{align*}
    &u\geq v \text{ on } \p\Omega\times(t_1,t_2] \text{ and on } \Omega \times \{t_1\}\\
    \intertext{and}
    &u \ge v, \quad \text{a.e. in } \bigl(\R^d\setminus \Omega \bigr) \times (t_1,t_2],
\end{align*}
then
\[
    u \ge v \quad \text{in } \Omega \times (t_1,t_2].
\]
\end{theorem}

\begin{proof}
The argument to obtain this result is the parabolic analog of \cite[Theorem 4.1]{KKL}. The subtle difference is in the treatment of the limits with respect to the time variable when we are outside $\Omega$. However, we overcome this difficulty by assuming $u,v \in C_{\rm loc}\bigl(t_1,t_2;L^{p-1}_{sp}(\mathbb R^d)\bigr)$ in the Definition \ref{def:viscosity-solution}. 

Set
\[
    M\coloneqq \sup_{\Omega\times(t_1,t_2]}(v-u).
\]
Assume, seeking a contradiction, that \(M>0\). By the parabolic boundary ordering, there exists
\[
    (\bar x,\bar t)\in\Omega\times(t_1,t_2]
\]
such that
\[
    M = v(\bar x, \bar t)-u(\bar x, \bar t)>0.
\]
For \(\varepsilon>0\), define
\[
    M_\varepsilon
    \coloneqq
    \sup_{\substack{x,y\in\overline\Omega\\ t,\tau\in[t_1,t_2]}}
    \left[
        v(x,t)-u(y,\tau)
        -\frac{|x-y|^2}{2\varepsilon}
        -\frac{|t-\tau|^2}{2\varepsilon}
    \right].
\]
Let $(x_\varepsilon,y_\varepsilon,t_\varepsilon,\tau_\varepsilon)$ be a point at which \(M_\varepsilon\) is attained. It then follows that $M_\varepsilon\ge M>0$. After passing to a subsequence,
\[
    x_\varepsilon,y_\varepsilon\to x^\ast,
    \qquad
    t_\varepsilon,\tau_\varepsilon\to t^\ast,
\]
and
\[
    v(x_\varepsilon,t_\varepsilon) - u(y_\varepsilon,\tau_\varepsilon) \to M.
\]
In particular, for \(\varepsilon\) small enough,
\[
    x_\varepsilon,y_\varepsilon\in\Omega,
    \qquad
    t_\varepsilon,\tau_\varepsilon>t_1.
\]
Define
\[
    \phi_\varepsilon(x,t)
    \coloneqq
    u(y_\varepsilon,\tau_\varepsilon)
    +M_\varepsilon
    +\frac{|x-y_\varepsilon|^2}{2\varepsilon}
    +\frac{|t-\tau_\varepsilon|^2}{2\varepsilon},
\]
and
\[
    \psi_\varepsilon(y,\tau)
    \coloneqq
    v(x_\varepsilon,t_\varepsilon)
    -M_\varepsilon
    -\frac{|y-x_\varepsilon|^2}{2\varepsilon}
    -\frac{|\tau-t_\varepsilon|^2}{2\varepsilon}.
\]
Then \(\phi_\varepsilon\) touches \(v\) from above at \((x_\varepsilon,t_\varepsilon)\), and \(\psi_\varepsilon\) touches \(u\) from below at \((y_\varepsilon,\tau_\varepsilon)\). Therefore, by
Lemma \ref{l:pointwise-equation}, we have
\[
    \partial_t\phi_\varepsilon(x_\varepsilon,t_\varepsilon)
    +
    (-\Delta_p)^s v(x_\varepsilon,t_\varepsilon)
    \le0,
\]
and
\[
    \partial_\tau\psi_\varepsilon(y_\varepsilon,\tau_\varepsilon)
    +
    (-\Delta_p)^s u(y_\varepsilon,\tau_\varepsilon)
    \ge0.
\]
Subtracting and using the equality of the time derivatives, we obtain
\begin{align}\label{eq:chiba}
    0
    \ge
    (-\Delta_p)^s v(x_\varepsilon,t_\varepsilon)
    -
    (-\Delta_p)^s u(y_\varepsilon,\tau_\varepsilon) = \int_{\mathbb R^d}
    \Theta_\varepsilon(z)|z|^{-d-sp}\,dz,
\end{align}
where
\[
\begin{aligned}
    \Theta_\varepsilon(z)
    &\coloneqq
    J_p\bigl(
        v(x_\varepsilon,t_\varepsilon)
        -
        v(x_\varepsilon+z,t_\varepsilon)
    \bigr)
    -
    J_p\bigl(
        u(y_\varepsilon,\tau_\varepsilon)
        -
        u(y_\varepsilon+z,\tau_\varepsilon)
    \bigr).
\end{aligned}
\]
For \(a\in\Omega\), set
\[
    E_a\coloneqq \{z\in\mathbb R^d \colon a+z\in\Omega\}.
\]
We decompose
\[
    \mathbb R^d
    =
    (E_{x_\varepsilon}\cap E_{y_\varepsilon})
    \cup
    \left(\mathbb R^d\setminus(E_{x_\varepsilon}\cap E_{y_\varepsilon})\right)
    \eqqcolon
    F_\varepsilon\cup F^\varepsilon.
\]
If \(z\in F_\varepsilon\), then both \(x_\varepsilon+z\) and \(y_\varepsilon+z\) belong to \(\Omega\). Comparing the maximum point with $(x_\varepsilon+z,y_\varepsilon+z,t_\varepsilon,\tau_\varepsilon)$,
we get $W_\varepsilon(z) \ge 0$, where
\[
\begin{aligned}
    W_\varepsilon(z)
    &\coloneqq
    v(x_\varepsilon,t_\varepsilon) - v(x_\varepsilon+z,t_\varepsilon)
    -
    u(y_\varepsilon,\tau_\varepsilon) + u(y_\varepsilon+z,\tau_\varepsilon).
\end{aligned}
\]
By Lemma \ref{l:Jp_est},
\[
    \Theta_\varepsilon(z)
    \ge
    c\left(
        |u(y_\varepsilon,\tau_\varepsilon)
        -
        u(y_\varepsilon+z,\tau_\varepsilon)|
        +
        W_\varepsilon(z)
    \right)^{p-2}
    W_\varepsilon(z) \geq 0 \quad \text{for } z\in F_\varepsilon.
\]
Hence
\[
    \int_{F_\varepsilon}
    \Theta_\varepsilon(z)|z|^{-d-sp}\,dz
    \ge0.
\]
It remains to estimate the exterior part \(F^\varepsilon\). Choose
\[
    \rho\coloneqq \frac14\operatorname{dist}(x^\ast,\partial\Omega)>0.
\]
For \(\varepsilon\) sufficiently small, $x_\varepsilon,y_\varepsilon\in B_\rho(x^\ast)$, and therefore $F^\varepsilon\subset B_{2\rho}^c$. Thus, the kernel is nonsingular on \(F^\varepsilon\). We split
\[
\begin{aligned}
    \int_{F^\varepsilon}
    \Theta_\varepsilon(z)|z|^{-d-sp}\,dz
    &=
    \int_{F^\varepsilon}
    \Theta_\varepsilon(z)
    \mathcal X_{\{W_\varepsilon<0\}}
    |z|^{-d-sp}\,dz
    \\
    &\quad+
    \int_{F^\varepsilon}
    \Theta_\varepsilon(z)
    \mathcal X_{\{W_\varepsilon\ge0\}}
    |z|^{-d-sp}\,dz.
\end{aligned}
\]
On \(\{W_\varepsilon\ge0\}\), monotonicity of \(J_p\) gives $\Theta_\varepsilon(z) \ge 0$. Therefore
\[
    \liminf_{\varepsilon\to0}
    \int_{F^\varepsilon}
    \Theta_\varepsilon(z)
    \mathcal X_{\{W_\varepsilon\ge0\}}
    |z|^{-d-sp}\,dz
    \ge0.
\]
We now prove that the negative part vanishes. Passing to a subsequence if
necessary, we may assume
\[
    \mathcal X_{F^\varepsilon}
    \to
    \mathcal X_{\mathbb R^d\setminus E_{x^\ast}}
    \qquad\text{a.e. in } \mathbb{R}^d. 
\]
The assumption that $u_-,v_+ \in C_{\rm loc}\bigl(t_1,t_2;L^{p-1}_{sp}(\mathbb R^d)\bigr)$, see \cite[Lemma 4.1]{KKL}, gives
\[
    \liminf_{\varepsilon\to0}
    u(y_\varepsilon+z,\tau_\varepsilon)
    \ge
    u(x^\ast+z,t^\ast),
\]
and
\[
    \limsup_{\varepsilon\to0}
    v(x_\varepsilon+z,t_\varepsilon)
    \le
    v(x^\ast+z,t^\ast),
\]
for a.e. \(z\in B_{2\rho}^c\). Since $v(x_\varepsilon,t_\varepsilon) - u(y_\varepsilon,\tau_\varepsilon) \to M$, we obtain
\[
    \liminf_{\varepsilon\to0} W_\varepsilon(z)
    \ge
    M+u(x^\ast+z,t^\ast)-v(x^\ast+z,t^\ast)
\]
for a.e. \(z\in B_{2\rho}^c\). On \(\mathbb R^d\setminus E_{x^\ast}\), we have
\(x^\ast+z\in\mathbb R^d\setminus\Omega\). Hence, the exterior ordering at time \(t^\ast\) gives
\[
    u(x^\ast+z,t^\ast)\ge v(x^\ast+z,t^\ast)
    \qquad\text{for a.e. }z\in\mathbb R^d\setminus E_{x^\ast}.
\]
Consequently,
\[
    \liminf_{\varepsilon\to0}
    W_\varepsilon(z)\mathcal X_{F^\varepsilon}(z)
    \ge
    M\mathcal X_{\mathbb R^d\setminus E_{x^\ast}}(z)
    \ge0
    \qquad\text{for a.e. } z \in \mathbb{R}^d.
\]
In particular, $\mathcal X_{\{W_\varepsilon<0\}}\mathcal X_{F^\varepsilon} \to 0$ a.e. in $\mathbb{R}^d$. We next dominate the negative part of \(\Theta_\varepsilon\). Since
\[
    v(x_\varepsilon,t_\varepsilon) - u(y_\varepsilon,\tau_\varepsilon) \to M,
\]
and \(v\) is upper semicontinuous, and \(u\) is lower semicontinuous near \((x^\ast,t^\ast)\), the quantities $v(x_\varepsilon,t_\varepsilon)$ and $u(y_\varepsilon,\tau_\varepsilon)$ are uniformly bounded. Elementary bounds for \(J_p\) therefore imply
\[
    \Theta_\varepsilon(z)
    \ge
    -C\left(
        1
        +
        v_+^{p-1}(x_\varepsilon+z,t_\varepsilon)
        +
        u_-^{p-1}(y_\varepsilon+z,\tau_\varepsilon)
    \right)
\]
for a.e. \(z\in F^\varepsilon\). By the assumption $u_-,v_+ \in C_{\rm loc}\bigl(t_1,t_2;L^{p-1}_{sp}(\mathbb R^d)\bigr)$ we have
\[
    v_+(x_\varepsilon+\cdot,t_\varepsilon) \to v_+(x^\ast+\cdot,t^\ast)
    \qquad 
    u_-(y_\varepsilon+\cdot,\tau_\varepsilon) \to u_-(x^\ast+\cdot,t^\ast)
\]
in $L^{p-1}_{sp}\bigl(B_{2\rho}^c\bigr)$. Since $\mathcal X_{\{W_\varepsilon<0\}}\mathcal X_{F^\varepsilon} \to 0$ a.e. in $\mathbb{R}^d$, the Dominated convergence theorem yields
\[
\begin{aligned}
    &\int_{F^\varepsilon}
    \left(
        1
        +
        v_+^{p-1}(x_\varepsilon+z,t_\varepsilon)
        +
        u_-^{p-1}(y_\varepsilon+z,\tau_\varepsilon)
    \right)
    \mathcal X_{\{W_\varepsilon<0\}}
    |z|^{-d-sp}\,dz
    \to0.
\end{aligned}
\]
Using the lower bound for \(\Theta_\varepsilon\), we conclude
\[
    \int_{F^\varepsilon}
    \Theta_\varepsilon(z)\mathcal X_{\{W_\varepsilon<0\}}
    |z|^{-d-sp}\,dz
    \to0.
\]
Thus
\[
    \liminf_{\varepsilon\to0}
    \int_{F^\varepsilon}
    \Theta_\varepsilon(z)|z|^{-d-sp}\,dz
    \ge0.
\]
Combining \eqref{eq:chiba} with the estimate on \(F_\varepsilon\), we obtain
\[
    \liminf_{\varepsilon\to0}
    \int_{F_\varepsilon}
    \Theta_\varepsilon(z)|z|^{-d-sp}\,dz
    \le0.
\]
But this can only happen if
\[
    \liminf_{\varepsilon\to0}W_\varepsilon(z)=0
    \qquad\text{for a.e. }z\in E_{x^\ast}.
\]
For such \(z\in E_{x^\ast}\), we have \(x^\ast+z\in\Omega\). By the upper semicontinuity of \(v\) and the lower semicontinuity of \(u\),
\[
    \liminf_{\varepsilon\to0}W_\varepsilon(z)
    \ge
    M-
    \bigl(
        v(x^\ast+z,t^\ast)-u(x^\ast+z,t^\ast)
    \bigr).
\]
Therefore
\[
    v(x^\ast+z,t^\ast)-u(x^\ast+z,t^\ast)
    \ge M
    \qquad\text{for a.e. }z\in E_{x^\ast}.
\]
Equivalently,
\[
    v(\xi,t^\ast)-u(\xi,t^\ast)\ge M
    \qquad\text{for a.e. }\xi\in\Omega.
\]
Since \(v-u\) is upper semicontinuous, this inequality holds for every \(\xi\in\Omega\).

Letting \(\xi\to\partial\Omega\) from inside and using the lateral boundary
ordering, we obtain
\[
    0
    \ge
    \limsup_{\Omega\ni \xi\to\xi_0}
    \bigl(v(\xi,t^\ast)-u(\xi,t^\ast)\bigr)
    \ge M>0,
\]
a contradiction. Hence \(M\le0\), and the claim follows. 
\end{proof}

\subsection{Weak implies viscosity}

We aim at proving that weak solutions to \eqref{eq:parabolic-fractional-p-laplacian} are also viscosity solutions. We will need the following result.

\begin{lemma}\label{l:classical-are-weak}
Let $w \in C^{1,1}(Q_r(x_0,t_0)) \cap L^{p-1}(t_0-r^2,t_0;L^{p-1}_{sp}(\R^d))$. If $\partial_t w + (-\Delta_p)^s w \geq 0$ in the pointwise sense in $Q_r(x_0,t_0)$, then $w$ is also a supersolution in the weak sense in $Q_r(x_0,t_0)$.
\end{lemma}

\begin{proof}
The proof of this lemma follows the same lines as \cite[Lemma 3.10]{KKL} and we skip it.    
\end{proof}

The following lemma establishes the continuity of the fractional Laplacian at smooth solutions.

\begin{lemma}\label{l:cont_frala}
Let $Q_r(x_0,t_0)\subset\Omega\times(t_1,t_2]$. Suppose $\phi\in  C_{\mathrm{loc}}(t_1,t_2;L^{p-1}_{sp}(\R^d))$ and for $(x,t)\in Q_r(x_0,t_0)$,
\[
    \phi(x,t)=\phi(x_0,t_0)+b(t-t_0)+\xi \cdot(x-x_0)+\frac{1}{2}(x-x_0)\cdot M(x-x_0).
\]
Then the function
\[
    (x,t)\to(-\Delta_p)^s\phi(x,t),
\]
is continuous at $(x_0,t_0)$.
\end{lemma}

\begin{proof}
The continuity in space is exactly \cite[Lemma 3.8]{KKL}. We only prove the continuity in time. Regarding the local term, since 
\[
    \phi(x_0, t)-\phi(y, t)=\phi(x_0, t_0)-\phi(y, t_0),
\]
we directly get
\begin{align*}
        \int_{B_r(x_0)}\frac{J_p(\phi(x_0, t)-\phi(y, t))-J_p(\phi(x_0,t_0)-\phi(y,t_0))}{|x_0-y|^{d+sp}}\dd y=0.
\end{align*}
Regarding the nonlocal term, we use the continuity of the tail. In fact,
\begin{align*}
     &  \int_{B^c_r(x_0)}\frac{|J_p(\phi(x_0,t)-\phi(y,t))-J_p(\phi(x_0,t_0)-\phi(y,t_0))|}{|x_0-y|^{d+sp}}\dd y\\
    &\leq C\int_{B^c_r(x_0)}\frac{ \left(  1+|\phi(y,t)|+|\phi(y,t_0)|  \right)^{p-2}|\phi(x_0,t_0)-\phi(x_0,t)|}{|x_0-y|^{d+sp}}\dd y\\
    &\quad + C\int_{B^c_r(x_0)}\frac{ \left(  1+|\phi(y,t)|^{p-2} \right)|\phi(y,t_0)-\phi(y,t)|}{|x_0-y|^{d+sp}}\dd y\\
    & \quad + C\int_{B^c_r(x_0)}\frac{ |\phi(y,t_0)-\phi(y,t)|^{p-1}}{|x_0-y|^{d+sp}}\dd y\\
    &\leq C_1\,\left( |\phi(x_0,t_0)-\phi(x_0,t)|+ \left(\int_{B^c_r(x_0)}\frac{ |\phi(y,t_0)-\phi(y,t)|^{p-1}}{|x_0-y|^{d+sp}}\dd y\right)^\frac{1}{p-1}\right),
\end{align*}
where $C_1$ depends further on $\sup_{t \in [t_0 - \bar r,t_0]}\|\phi(\cdot,t)\|_{L^{p-1}_{sp}(\R^d)}$. These two terms clearly go to zero as $t\to t_0$. 
\end{proof}

Now we prove that weak solutions are viscosity solutions.

\begin{theorem}\label{t:weak-are-viscosity}
Let $u$ be a weak subsolution to \eqref{eq:parabolic-fractional-p-laplacian} in $\Omega \times (t_1,t_2]$. If $u$ satisfies 
\[
    u \in C_{\mathrm{loc}}(\Omega \times (t_1,t_2]) \cap C_{\mathrm{loc}}(t_1,t_2;L^{p-1}_{sp}(\R^d)),
\]
then it is also a viscosity subsolution in $\Omega \times (t_1,t_2]$. A similar statement holds for supersolutions.    
\end{theorem}

\begin{proof}
We prove only the subsolution part. Let $\phi \in C^2(Q_r(x_0,t_0))$, for some cylinder $Q_r(x_0,t_0)\subset \Omega \times (t_1,t_2]$, and assume that $\phi$ touches $u$ from above at $(x_0,t_0)$ in $Q_r(x_0,t_0)$. We must show that
\[
    \partial_t \phi(x_0,t_0) + (-\Delta_p)^s \phi_r(x_0,t_0)\le 0,
\]
where
\[
    \phi_r(x,t)\coloneqq
    \begin{cases}
    \phi(x,t) & \text{if } (x,t)\in Q_r(x_0,t_0),\\[2mm]
    u(x,t) & \text{otherwise}.
\end{cases}
\]

Suppose, seeking a contradiction, that
\[
    \partial_t \phi(x_0,t_0) + (-\Delta_p)^s \phi_r(x_0,t_0)\geq \varepsilon> 0,
\]
for some $\varepsilon>0$.

Up to taking $r$ smaller, we can suppose that $\phi$ is a quadratic polynomial
\[
    \phi(x,t)=u(x_0,t_0)+b(t-t_0)+\xi\cdot(x-x_0)+\frac{1}{2}(x-x_0)\cdot M(x-x_0).
\]
Note that $\p_t\phi(x,t)-\p_t\phi(x_0,t_0) = 0$.

Using Lemma \ref{l:cont_frala}, we can find a small $0<\bar r<r$ such that 
\[
    \partial_t \phi + (-\Delta_p)^s \phi_r > \varepsilon/2 \quad \text{in} \quad Q_{\bar r}(x_0,t_0).
\] 
In this step we use assumption $u\in C_{\mathrm{loc}}(t_1,t_2;L^{p-1}_{sp}(\R^d)) $. 

We now develop the parabolic analog of \cite[Lemma 3.9]{KKL}. We can choose functions
\[
\eta \in C_0^2(B_{\bar r/2}(x_0)), \qquad 0 \le \eta \le 1, \qquad \eta(x_0)=1,
\]
and
\[
\psi \in C^1([t_0-\bar r^2,t_0]), \qquad 0 \le \psi \le 1, \qquad \psi(t_0-\bar r^2)=0, \qquad \psi(t_0)=1,
\]
such that the function \(\varphi_\delta(x,t)\), defined by
\[
    \varphi_\delta(x,t)=
    \begin{cases}
    \phi(x,t)-\delta\,\eta(x)\psi(t) & \text{if } (x,t)\in Q_r(x_0,t_0),\\[2mm]
    u(x,t) & \text{otherwise},
\end{cases}
\]
satisfies
\begin{equation}\label{small-perturbation-solution}
\sup_{Q_{\bar r}(x_0,t_0)}
\left| \bigl(\partial_t\phi + (-\Delta_p)^s\phi_r\bigr)-\bigl(\partial_t\varphi_\delta + (-\Delta_p)^s\varphi_\delta\bigr)\right|
< \frac{\varepsilon}{4},
\end{equation}
provided that \(\delta>0\) is sufficiently small. Indeed, by the regularity assumption on $\psi$, we have
\[
    \sup_{Q_{\bar r}(x_0,t_0)}|\partial_t \phi - \partial_t \varphi_\delta| \leq \delta \|\partial_t \psi\|_{L^\infty([t_0-\bar r^2,t_0])} \leq \varepsilon/8,
\]
for $\delta$ small enough. Now, given any $(x,t) \in Q_{\bar r}(x_0,t_0)$, we compute
\[
    \mathcal{I}(x,t) \coloneqq |(-\Delta_p)^s \phi_r(x,t) - (-\Delta_p)^s \varphi_\delta(x,t)|.
\]
By Lemma \ref{l:Jp_est}, we have that $\mathcal{I}(x,t)$ can be estimated by
\begin{align*}
    &\int_{B_{ r}(x_0)} \frac{|J_p(\phi(x,t) - \phi(y,t)) - J_p(\phi(x,t) - \phi(y,t) + \delta \psi(t)(\eta(y)-\eta(x)))|}{|y-x|^{d+sp}}\dd y\\
   & \quad + \int_{\R^d \backslash B_{ r}(x_0)} \frac{|J_p(\phi(x,t) - u(y,t)) - J_p(\phi(x,t) - \delta \psi(t)\eta(x) - u(y,t))|}{|y-x|^{d+sp}}\dd y\\
    \leq&\, C_p\int_{B_{ r}(x_0)} \frac{(|\phi(x,t) - \phi(y,t)| + \delta \psi(t)|\eta(x)-\eta(y)|)^{p-2}\delta \psi(t)|\eta(x)-\eta(y)|}{|y-x|^{d+sp}}\dd y\\
   & \quad + C_p \int_{\R^d \backslash B_{ r}(x_0)} \frac{(|\phi(x,t) - u(y,t)| + \delta \psi(t)\eta(x))^{p-2}\delta \psi(t)\eta(x)}{|y-x|^{d+sp}}\dd y\\
    \leq&\, C_p\int_{B_{ r}(x_0)} \frac{(|\phi(x,t) - \phi(y,t)| + \delta |\eta(x)-\eta(y)|)^{p-2}\delta |\eta(x)-\eta(y)|}{|y-x|^{d+sp}}\dd y\\
   & \quad + C_p \int_{\R^d \backslash B_{ r}(x_0)} \frac{(|\phi(x,t) - u(y,t)| + \delta \eta(x))^{p-2}\delta \eta(x)}{|y-x|^{d+sp}}\dd y,
\end{align*}
where we used that $0 \leq \psi \leq 1$ in the last inequality. 

Now we split $B_r(x_0)$ into $B_{\bar r}(x_0)$ and $B_r(x_0)\setminus B_{\bar r}(x_0)$. We can bound the integral over the annulus as follows. Taking $x\in \text{supp } \eta\subset B_{\bar r/2}(x_0)$ and using the boundedness of $\phi$, we get
\begin{align*}
    \int_{B_r(x_0)\setminus B_{\bar r}(x_0)} \frac{(|\phi(x,t) - \phi(y,t)| + \delta \eta(x))^{p-2}\delta \eta(x)}{|y-x|^{d+sp}}\dd y\leq C(\phi,\bar r)\delta,
\end{align*}
which is small for small $\delta$.

From here, following the same steps as in the proof of \cite[Lemma 3.9]{KKL}, we can choose $\delta$ small enough such that
\[
    \mathcal{I}(x,t) < \varepsilon/8, \quad \text{for all} \quad (x,t) \in Q_{\bar r}(x_0,t_0).
\]
This proves \eqref{small-perturbation-solution}, which further gives
\[
    \partial_t \varphi_\delta + (-\Delta_p)^s \varphi_\delta > \varepsilon/4 \quad \text{in} \quad Q_{\bar r}(x_0,t_0).
\]
We can then apply Lemma \ref{l:classical-are-weak} to obtain that $\varphi_\delta$ is also a weak supersolution in $Q_{\bar r}(x_0,t_0)$. We observe that since $\varphi_\delta \geq u$ outside $Q_{\bar r}(x_0,t_0)$, we can use Theorem \ref{t:comparison-weak} to obtain $\varphi_\delta \geq u$ in $Q_{\bar r}(x_0,t_0)$, but
\[
    u(x_0,t_0) - \delta = \phi(x_0,t_0) - \delta = \varphi_\delta(x_0,t_0) \geq u(x_0,t_0),
\]
which is a contradiction.
\end{proof}

\subsection{Viscosity implies weak}

We now prove that viscosity solutions are also weak solutions. 

We define the notion of inf-convolution. Let $u \in \mathrm{LSC}(\Omega \times (t_1,t_2]) \cap L^\infty(\R^d \times (t_1,t_2])$. Given a parameter $\varepsilon>0$, we define
\[
    u_\varepsilon(x,t) \coloneqq \inf_{(y,\tau) \in \R^d \times (t_1,t_2]}\left\{u(y,\tau) + \frac{|y-x|^2 + |t-\tau|}{\varepsilon} \right\}.
\]
Note that $u_\varepsilon$ is well-defined because $u \in L^\infty(\R^d \times (t_1,t_2])$, and it is $C^{1,1}$ in space from above and $C^{0,1}$ in time. A useful property of $u_\varepsilon$ is that if $u$ is a supersolution in $\Omega \times (t_1,t_2]$, then $u_\varepsilon$ is a supersolution in
\begin{equation}\label{eq:converging-domain}
    \mathcal{P}_\varepsilon \coloneqq \Omega_\varepsilon \times (t_1+2\varepsilon\|u\|_\infty,t_2], \quad \text{where} \quad \Omega_\varepsilon \coloneqq \{x \in \Omega \colon B_{\sqrt{2\varepsilon \|u\|_{L^\infty}}}(x) \subset \Omega \}.
\end{equation}

\begin{lemma}\label{l:correios-telegrafos-telefones}
Let $u$ be a viscosity supersolution to \eqref{eq:parabolic-fractional-p-laplacian} in $\Omega \times (t_1,t_2]$ such that $u \in L^\infty(\R^d \times (t_1,t_2])$. Then, $u_\varepsilon$ is a viscosity supersolution in $\mathcal{P}_\varepsilon$. In particular, $u_\varepsilon$ solves the equation almost everywhere in $\mathcal{P}_\varepsilon$.
\end{lemma}

\begin{proof}
Let $\varphi \in C^2(Q_r(x_0,t_0))$ touch $u_\varepsilon$ from below in $Q_r(x_0,t_0)$ at $(x_0,t_0)$. We want to prove that
\[
    \psi(x,t) \coloneqq
    \begin{cases}
    \varphi(x,t) & \text{if } (x,t)\in Q_r(x_0,t_0),\\[2mm]
    u_\varepsilon(x,t) & \text{otherwise},
    \end{cases}
\]
satisfies
\begin{equation}\label{estimate-visc-u_ep}
    \partial_t \psi(x_0,t_0) + (-\Delta_p)^s \psi(x_0,t_0) \geq 0.
\end{equation}
Let $(x_\varepsilon, t_\varepsilon)$ be the point satisfying
\[
    u_\varepsilon(x_0,t_0) = u(x_\varepsilon, t_\varepsilon) + \underbrace{\frac{|x_0- x_\varepsilon|^2 + |t_0-  t_\varepsilon|}{\varepsilon}}_{\coloneqq c_\varepsilon},
\]
and define
\[
    \overline{\varphi}(x,t) \coloneqq \varphi(x+x_0 -  x_\varepsilon, t + t_0 - t_\varepsilon) - c_\varepsilon.
\]
It then follows that $\overline{\varphi}$ touches $u$ from below in $Q_r(x_\varepsilon, t_\varepsilon)$ at $( x_\varepsilon, t_\varepsilon)$. Since $u$ is a viscosity supersolution to \eqref{eq:parabolic-fractional-p-laplacian}, we obtain that
\[
    \overline \psi(x,t) \coloneqq
    \begin{cases}
    \bar \varphi(x,t) & \text{if } (x,t)\in Q_r(x_\varepsilon,t_\varepsilon),\\[2mm]
    u(x,t) & \text{otherwise},
    \end{cases}
\]
satisfies
\[
    \partial_t \overline \psi(x_\varepsilon,t_\varepsilon) + (-\Delta_p)^s \overline \psi(x_\varepsilon,t_\varepsilon) \geq 0.
\]
To obtain \eqref{estimate-visc-u_ep}, it is enough to prove $(-\Delta_p)^s \overline \psi(x_\varepsilon,t_\varepsilon) \leq (-\Delta_p)^s \psi(x_0,t_0)$, since $\partial_t \overline \psi(x_\varepsilon,t_\varepsilon) = \partial_t \psi(x_0,t_0)$.

We note that for $x \in B_r(x_\varepsilon)$, we have
\[
    \overline \psi(x_\varepsilon,t_\varepsilon) - \overline \psi(x,t_\varepsilon) = \varphi(x_0,t_0) - \varphi(x+x_0- x_\varepsilon,t_0).
\]
Moreover, by definition of $u_\varepsilon$ and $\bar \varphi$, we obtain
\[
    \bar \varphi(x_\varepsilon,t_\varepsilon) - u_\varepsilon(x,t_0) \leq u_\varepsilon(x_0,t_0) - u_\varepsilon(x + x_0 - x_\varepsilon,t_0), \quad \text{for} \quad x \in \R^d.
\]
Therefore, by monotonicity of $J_p$, we obtain after changing variables in the integral,
\[
    (-\Delta_p)^s \overline\psi(x_\varepsilon,t_\varepsilon) \leq (-\Delta_p)^s \psi(x_0,t_0).
\]
This concludes the first part of the proof. The fact that the inf-convolution solves the equation almost everywhere is a standard property of these mollifiers; see, for example, \cite[Lemmas 4.3 and 5.7]{CS}.
\end{proof}

We now have the following Caccioppoli-type estimate, which follows as a consequence of \cite[Corollary 2.1]{L}.

\begin{lemma}\label{l:sha-shi}
Let $u$ be a  weak supersolution to \eqref{eq:parabolic-fractional-p-laplacian} in $\Omega\times (t_1,t_2]$ satisfying
\[
    \|u\|_{L^\infty(\R^d\times(t_1,t_2])}\leq M,
\]
for some $M>0$. Then for all $\mathcal{K}\times(\tau_1,\tau_2]\Subset \Omega\times (t_1,t_2]$ we have
\begin{align}\label{eq:cup_fe}
    \int_{\tau_1}^{\tau_2}\iint_{\mathcal{K}\times \mathcal{K}}\frac{|u(x,t)-u(y,t)|^{p}}{|x-y|^{d+sp}}\dd x\dd y \dd t\leq C,
\end{align}
where $C>0$ depends on $\mathcal{K}\times(\tau_1,\tau_2],\Omega\times (t_1,t_2], M$ and universal constants.
\end{lemma}

\begin{proof}
By \cite[Corollary 2.1]{L}, we obtain, for any $Q \coloneqq Q_{2r}(x_0,t_0)\subset \Omega\times (t_1,t_2]$,
    \begin{align*}
& \int_{t_0 -r^2}^{t_0} \int_{B_r(x_0)} \int_{B_r(x_0)}
\frac{\lvert u(x,t) - u(y,t) \rvert^p}
{\lvert x - y \rvert^{d+sp}}
\dd x\dd y\dd t \\
&\le 
\gamma r^{-sp}
\iint_{Q} u^p(x,t)\dd x\dd t  
+ \gamma r^{-sp} M^{p-1}
\iint_{Q} |u(x,t)|\dd x\dd t
 \\
&\quad
+ \gamma r^{-2}
\iint_{Q} u^2(x,t)\dd x\dd t \\
&\le C(r,M).
\end{align*}

Now, for any arbitrary set $\mathcal{K}\times(\tau_1,\tau_2]\Subset \Omega\times (t_1,t_2]$, we can use a standard covering argument (see, for example, \cite[Lemma 5]{KKP}) to obtain \eqref{eq:cup_fe}.
\end{proof}

The next result states that if $u$ can be approximated by a sequence of weak supersolutions, then it must be itself a weak supersolution.

\begin{proposition}\label{p:brez-bunk}
Let $(u_j)_{j \in \mathbb{N}}$ be a sequence of weak supersolutions to \eqref{eq:parabolic-fractional-p-laplacian} such that
\[
    \|u_j\|_{L^\infty(\R^d \times (t_1,t_2])} \leq M,  \quad \text{for all } j \in \mathbb{N}.
\]
Suppose that $u_j$ converges to a function $u$ almost everywhere in $\R^d \times (t_1,t_2]$ as $j\to \infty$. Then, $u$ is a weak supersolution in $\Omega \times (t_1,t_2]$ as well.
\end{proposition}

\begin{proof}
Let $\mathcal{K}\times [\tau_1, \tau_2] \Subset \Omega\times (t_1,t_2]$ and $\varphi$ be a nonnegative function such that
\[
   \varphi \in L^p(\tau_1,\tau_2;W^{s,p}_0(\mathcal{K})) \cap H^1(\tau_1,\tau_2;L^2(\mathcal{K})).
\]
By the definition of weak supersolution, we have
\begin{equation}\label{eq:liv_espae}
\begin{aligned}
    &-\int_{\mathcal{K}}u_j(x,t)\varphi(x,t)\dd x \bigg|_{\tau_1}^{\tau_2} + \int_{\tau_1}^{\tau_2}\int_{\mathcal{K}}u_j(x,t) \partial_t \varphi(x,t)\dd x\dd t \\
    \leq &\int_{\tau_1}^{\tau_2} \mathcal{E}(u_j(\cdot,t),\varphi(\cdot,t))\dd t \\
    =& \int_{\tau_1}^{\tau_2} \mathcal{E}(u(\cdot,t),\varphi(\cdot,t))\dd t + \int_{\tau_1}^{\tau_2}\Big[\mathcal{E}(u_j(\cdot,t),\varphi(\cdot,t)) -  \mathcal{E}(u(\cdot,t),\varphi(\cdot,t))\Big]\dd t,
\end{aligned}
\end{equation}
where the quantity $\mathcal{E}(v(\cdot,t),\varphi(\cdot,t))$ is defined as
\[
    \iint_{\R^d \times \R^d} \frac{J_p(v(x,t) - v(y,t))(\varphi(x,t) - \varphi(y,t))}{|x-y|^{d+sp}}\dd y\dd x. 
\]
Now, we split $\mathcal{E}(u_j(\cdot,t),\varphi(\cdot,t)) -  \mathcal{E}(u(\cdot,t),\varphi(\cdot,t))$ as follows
\begin{align*}
    & \iint_{\mathcal{K} \times \mathcal{K}} \frac{\Big[J_p(u_j(x,t) - u_j(y,t)) - J_p(u(x,t) - u(y,t))\Big](\varphi(x,t) - \varphi(y,t))}{|x-y|^{d+sp}}\dd y\dd x\\
    & + 2\int_\mathcal{K}\int_{\mathcal{K}^c } \frac{\Big[J_p(u_j(x,t) - u_j(y,t)) - J_p(u(x,t) - u(y,t))\Big]\varphi(x,t)}{|x-y|^{d+sp}}\dd y\dd x\\
    & = A_{1,j}(t) + 2A_{2,j}(t).
\end{align*}
We wish to take the limit as $j\to \infty$ in equation \eqref{eq:liv_espae}. By the dominated convergence theorem, we obtain that the LHS converges to
\[
    -\int_{\mathcal{K}}u(x,t)\varphi(x,t)\dd x \bigg|_{\tau_1}^{\tau_2} + \int_{\tau_1}^{\tau_2}\int_{\mathcal{K}}u(x,t) \partial_t \varphi(x,t)\dd x\dd t.
\]
The proof follows by showing that
\[
    \lim_{j \to \infty} \int_{\tau_1}^{\tau_2} \left(A_{1,j}(t) + 2A_{2,j}(t) \right) = 0,
\]
and this can be done by reasoning as in the proof of \cite[Theorem 10]{KKP}, using the bound from Lemma \ref{l:sha-shi}.
\end{proof}

We conclude this section with the implication that viscosity solutions are also weak solutions.

\begin{theorem}\label{t:ind_imp_phip}
Let $u$ be a viscosity supersolution to \eqref{eq:parabolic-fractional-p-laplacian} in $\Omega \times (t_1,t_2]$. If 
\[
    u  \in L^\infty(\R^d \times (t_1,t_2]) \cap \mathrm{LSC} (\R^d \times (t_1,t_2]),
\]
then $u$ is a weak supersolution as well. A similar statement holds for subsolutions.
\end{theorem}

\begin{proof}
To prove the supersolution property, let $Q \Subset \Omega \times (t_1,t_2]$, and $\varepsilon>0$ be small enough such that $Q \subset \mathcal{P}_\varepsilon$, where $\mathcal{P}_\varepsilon$ is given by \eqref{eq:converging-domain}. 

Let $u_\varepsilon$ be the inf-convolution of $u$, which solves the same equation almost everywhere in $Q \subset \mathcal{P}_\varepsilon$ by Lemma \ref{l:correios-telegrafos-telefones}. Now, we use Lemma \ref{l:classical-are-weak} to obtain that $u_\varepsilon$ is a weak supersolution to \eqref{eq:parabolic-fractional-p-laplacian} in $Q$ for every $\varepsilon$ small enough. By standard properties of the inf-convolution, we obtain that $u_\varepsilon \to u$ almost everywhere in $\R^d \times (t_1,t_2]$ as $\varepsilon \to 0$. We then use Proposition \ref{p:brez-bunk} to obtain that $u$ is a weak supersolution to \eqref{eq:parabolic-fractional-p-laplacian} in $Q$. The proof is done once we recall that $Q$ is arbitrary.

Regarding the subsolution property, we need to consider sup-convolutions, which are defined as $u^\varepsilon \coloneqq -(-u)_\varepsilon$. They satisfy properties similar to those of the inf-convolution and approximate $u$ from above. The very same reasoning applies to obtain that $u$ is a weak subsolution.
\end{proof}

\section{Parabolic Lipschitz regularity}\label{s:lipschitz-regularity}

In this section, we establish a Lipschitz estimate in the spatial variables for solutions of \eqref{eq:parabolic-fractional-p-laplacian}, thereby extending to the parabolic setting the corresponding results in \cite{BT} and \cite{BS}.

For a viscosity solution $u$ to \eqref{eq:parabolic-fractional-p-laplacian}, we crucially assume that 
\[
    u \in C_{\rm loc}(-1,0;L^{p-1}_{sp}(\R^d)),
\]
and work in the normalized setting
\begin{equation}\label{assumption:normalized-setting}
    \|u\|_{L^\infty(Q_1)} + \sup_{t \in (-1,0]}\|u(\cdot,t)\|_{L^{p-1}_{sp}(\R^d)} \leq 1.
\end{equation}

Moreover, we shall use the following bootstrap hypothesis. For some
\(\kappa\in(0,1)\), assume that there exist constants \(C_1(\kappa)\) and
\(C_2\) such that, for every \(x,y\in B_1\) and \(t,\tau\in(-1,0]\),
\begin{equation}\label{assumption:bat-ind}
    |u(x,t)-u(y,\tau)|
    \leq
    C_1(\kappa)|x-y|^\kappa
    +
    C_2|t-\tau|^{\kappa_0}.
\end{equation}
The initial exponent \(\kappa=\kappa_0\) is provided by Theorem
\ref{t:NL_Bra_nobo} for weak solutions, which, using the implication in Theorem \ref{t:weak-are-viscosity}, are also viscosity solutions, and therefore the viscosity argument below applies to them. The Ishii--Lions method will be used to improve the spatial exponent in \eqref{assumption:bat-ind}; iterating this improvement yields the desired
spatial Lipschitz estimate.

\subsection{The setup}

Our aim is to prove that if $u$ is a viscosity solution to \eqref{eq:parabolic-fractional-p-laplacian} in $Q_1$, then there exists a universal constant $L$ such that
\[
    \|\nabla u\|_{L^\infty(Q_{1/2})} \leq L.
\]
For that purpose, we set up a parabolic Ishii-Lions type argument. 

For constants $L, L_2, \beta^\ast$ to be fixed, and $t_0 \in (-1,0]$, define the function with doubled spatial variables
\[
    \Phi(x,y,t) \coloneqq u(x,t)-u(y,t)-L\,\omega(\abs{x-y})-L_2 \,\psi(x) - L_2(t_0-t)^{1+\beta^\ast},
\]
for $x,y \in B_1$ and $t \in (-1,t_0]$, where $\omega$ is either a H\"older or Lipschitz modulus of continuity, to be defined later, and $\psi(x)\coloneqq \psi_0(x)^m$, $m>2$, for a smooth nonnegative $\psi_0$ that vanishes on $\overline{B}_{1/2}$ and is strictly positive on $B_1\setminus \overline{B}_{1/2}$. This choice implies that there exists a constant $C>0$ such that
\begin{align}\label{eq:inequality_psi}
    |\nabla \psi(x)|\leq C\psi(x)^\frac{m-1}{m}, \quad \text{ for all } x\in B_1.
\end{align}
We will prove that, for a choice of $L_2$, $m$, and $\beta^\ast$ large and universal, there exists $L$ large enough such that the function $\Phi$ is nonpositive, which implies that $u$ has modulus of continuity $\omega$ in space at $t = t_0$. For the sake of simplicity, we assume $t_0 = 0$.

Seeking a contradiction, we assume there exists a triplet $(\bar x, \bar y, \bar t)$ such that
\begin{align*}
    \Phi(\bar x, \bar y, \bar t) = \sup_{x,y \in \overline{B_1} \text{ and } t \in [-1,0]} \Phi(x,y,t) = M_L>0.
\end{align*}
Now we double the time variable as well. Define
\[
\begin{aligned}
\Psi_K(x,y,t,\tau)\coloneqq{}&
u(x,t)-u(y,\tau)-L\,\omega(\abs{x-y})-L_2\psi(x) \\
&\qquad -L_2(t_0-t)^{1+\beta^\ast}-K(t-\tau)^2.
\end{aligned}
\]
Then the following must hold 
\begin{align}\label{eq:IL_contradiction}
    M_K \coloneqq \sup_{x,y \in \overline{B_1} \text{ and } t,\tau \in [-1,0]} \Psi_K(x,y,t,\tau) \geq M_L > 0.
\end{align}
We moreover define
\[
    a_K=x_K-y_K,\quad b_K \coloneqq t_K-\tau_K.
\]
From the above considerations, we immediately get
\[
    L\omega(|a_K|) + L_2 \psi(x_K) + L_2(-t_K)^{1+\beta^\ast} + K(t_K-\tau_K)^2 < 2.
\]
By making $L_2$ large, we force $x_K\in B_{5/8}$ and $|t_K|<1/4$. From $\omega(|a_K|)<2/L$, we see that $|a_K|\to 0$ as $L\to\infty$ and therefore $y_K\in B_{6/8}$ by taking $L$ large. Finally, from $K(t_K-\tau_K)^2 < 2$, we get $|b_K|\to 0$ as $K\to \infty$ and so we can take $K$ large so that $|\tau_K|<1/2$.

Since $M_K\geq M_L>0$ we also know that, up to subsequence $a_K\to \tilde a$ as $K\to \infty$ for some $\tilde a$ satisfying $|\tilde a|>0$. 

In the argument of this section, we will consider $L$ to be a constant fixed large enough to satisfy some conditions which will appear throughout the section, but which is independent of $K$. Then, we take $K$ large enough so that $|a_K|>|\tilde a|/2$. Note that $\tilde a$ depends implicitly on $L$, which is fixed. We can now take $K$ possibly even larger to ensure the following
\begin{align}\label{eq:InshSil}
    |b_K|^{\kappa_0} + &\left(\displaystyle \int_{B_{\frac{1}{16}}^c} |u(y_K+z, t_K)-u(y_K+z, \tau_K)|^{p-1}\frac{\dd z}{|z|^{d+sp}}\right)^\frac{1}{p-1}\\
    &\leq \frac{|\tilde a|}{2}\leq |a_K|.
\end{align}
Therefore, $K$ depends on $L$, but this dependence is harmless. Moreover, $K$ depends on the modulus of continuity of the $\Tail$ of $u$. But since $L$ does not depend on $K$, this dependence is harmless as well.

Define the auxiliary function 
\begin{equation}\label{eq:phi-def}
\phi(x,y) \coloneqq  L\omega(|x-y|) + L_2 \psi(x).
\end{equation}

In order to obtain viscosity inequalities, we define the following test functions 
\begin{align*}
w_1(x,t)
&=
\begin{cases}
\Lambda_1 + \phi(x,y_K) + L_2|t|^{1+\beta^\ast} + K(t-\tau_K)^2
& \text{for } x\in B_{\delta_1\,|a_K|}(x_K),\\
u(x,t)
& \text{otherwise},
\end{cases}
\\[1em]
w_2(y,\tau)
&=
\begin{cases}
\Lambda_{2} - \phi(x_K,y) - K(t_K-\tau)^2
& \text{for } y\in B_{\delta_1\,|a_K|}(y_K),\\
u(y,\tau)
& \text{otherwise},
\end{cases}
\end{align*}
where $\Lambda_{1} = u(y_K,\tau_K) + M_K$ and $\Lambda_{2} = u(x_K,t_K) - M_K -L_2|t_K|^{1+\beta^\ast}$. The constant $\delta_1$ will be chosen to be small, possibly depending on $|a_K|$. The time variables $t$ and $\tau$ are defined in a neighborhood of $t_K$ and $\tau_K$, respectively.

Note that $w_1$ touches $u$ from above at $(x_K,t_K)$ and $w_2$ touches $u$ from below at $(y_K,\tau_K)$. Therefore, by the viscosity condition, we have the following inequalities
\begin{align*}
     &\p_t w_1(x_K,t_K) + (-\Delta_p)^s w_1(x_K,t_K) \leq \, 0,\\
     &\p_t w_2(y_K,\tau_K) + (-\Delta_p)^s w_2(y_K, \tau_K) \geq \, 0.
\end{align*}
Subtracting the inequalities, we obtain 
\begin{equation}\label{eq:evaluate-at-max}
   \begin{aligned}
       \II\,&\coloneqq(-\Delta_p)^s w_1(x_K,t_K) - (-\Delta_p)^s w_2(y_K, \tau_K) \\
       &\leq \p_tw_2(y_K,\tau_K)-\p_tw_1(x_K,t_K)= L_2(1+\beta^\ast)(-t_K)^{\beta^\ast}.
   \end{aligned}
\end{equation}

For the sake of organization, we will split the analysis of the term $\II$ and the time derivative part into several lemmas, and put all of them together in Section \ref{ss:ael_tomuma_noiist}  to get a contradiction in the different scenarios, with the fact that we are assuming \eqref{eq:IL_contradiction}.

The next lemma gives a bound for  $\p_tw_2(y_K,\tau_K)-\p_tw_1(x_K,t_K)$, using the already known regularity. 

\begin{lemma}\label{l:estimate_time}
Let $u$ be a viscosity solution to \eqref{eq:parabolic-fractional-p-laplacian}, and assume \eqref{assumption:normalized-setting} and \eqref{assumption:bat-ind}. Let $K$ be large enough so that \eqref{eq:InshSil} holds.  Then, the quantity $\II$ defined in \eqref{eq:evaluate-at-max} satisfies
\begin{equation}\label{eq:decay-I-term}
        \II \leq C(\kappa)|a_K|^\frac{\kappa \beta^\ast}{1+\beta^\ast},
\end{equation}    
where $C>0$ depends on $L_2$, $\beta^\ast$, $\kappa$, and universal constants.
\end{lemma}

\begin{proof} 
From the previous discussion, we know that the point where the supremum \eqref{eq:IL_contradiction} is attained is interior. Since the supremum in \eqref{eq:IL_contradiction} is positive, we have
\[
    L_2(-t_K)^{1+\beta^\ast} < u(x_K,t_K) - u(y_K,\tau_K).
\]
Combining it with assumption \eqref{assumption:bat-ind}, we obtain
\[
    L_2(-t_K)^{1+\beta^\ast} < C(\kappa)|a_K|^\kappa +C_1|b_K|^{\kappa_0}.
\]
In view of \eqref{eq:evaluate-at-max}, we obtain
\begin{align*}
   \II & \leq L_2(1+\beta^\ast)|t_K|^{\beta^\ast}\\
   & = (1+\beta^\ast)L_2^{\frac{1}{1+\beta^\ast}} C^{\frac{\beta^\ast}{1+\beta^\ast}} \left(C(\kappa)|a_K|^\kappa +C_1|b_K|^{\kappa_0}\right)^\frac{\beta^\ast}{1+\beta^\ast}.
\end{align*}
Using that $|b_K|^{\kappa_0}\leq |\tilde a|/2\leq |a_K|\leq |a_K|^\kappa$, we conclude.
\end{proof}

\subsection{The cone of concavity}Take $a \in B_{1/2}$, and let $\delta_0$ be a very small number, to be fixed later, possibly depending on $a$. We define the cone $\C(a)$ as
\begin{equation} \label{def:cone}
    \C(a) \coloneqq \left\{ z \in B_{|a|/2} \colon |\langle a, z \rangle| \geq \sqrt{1 - \delta_0^2} \, |a|\,|z| \right\}.
\end{equation}

We will consider two different moduli of continuity, the first one of H\"older type and then one of Lipschitz type, modified by a lower order term to make it strictly concave. We denote them respectively by
\begin{align}\label{eq:Holder_omega}
    \omega_\gamma(r) \coloneqq r^\gamma,
\end{align}
and
\begin{align}\label{eq:Lipschitz_omega}
    \widetilde\omega(r) \coloneqq r + \frac{r}{20 \log (r/4)}.
\end{align}

We use the notation $\delta^2 f$ to denote the second increment of a generic function $f$ as
\[
    \delta^2 f(x,z) \coloneqq f(x) + \nabla f(x) \cdot z - f(x + z).
\]
In particular, when we write $\delta^2 \omega(|\cdot|)(a,z)$, we mean
\[ 
    \delta^2 \omega(|\cdot|)(a,z) = \omega(|a|) + \omega'(|a|) \, \frac {a}{|a|} \cdot z - \omega(|a+z|). 
\]
We start by bounding the second incremental quotient of the moduli of continuity.

\begin{lemma} \label{l:2nd-increment-omega}
Let $a \in B_{1/2}$ and $z \in \C(a)$. Assume that $\delta_0\leq \bar c_0$ for some small universal constant $\bar c_0$. Then the following estimates hold
\begin{equation} \label{eq:2nd-increment-Holder}        c|a|^{\gamma-2}|z|^2\leq \delta^2 \omega_\gamma(|\cdot|)(a,z) \leq C|a|^{\gamma-2}|z|^2
\end{equation}
and
\begin{equation} \label{eq:2nd-increment-Lipschitz} \frac{c}{|a| \log^2 |a|} |z|^2 \leq \delta^2 \widetilde\omega(|\cdot|)(a,z) \leq \frac{C}{|a| \log^2 |a|} |z|^2,
\end{equation}
where $c$ and $C$ are universal constants.
\end{lemma}

\begin{proof}
Since \eqref{eq:2nd-increment-Lipschitz} is exactly \cite[Lemma 6.4]{GJS}, we only prove \eqref{eq:2nd-increment-Holder}. When $a$ and $z$ are colinear, we can use a Taylor expansion for $\omega_\gamma$ to write
\[
    \delta^2 \omega_\gamma(|\cdot|)(a,z)= -\frac{1}{2}\omega_\gamma''(\xi)\,|z|^2,
\]
for some $\xi$ in the line segment $[a,a+z]\subset[a/2, 3a/2]$, where we have the estimate $\omega_\gamma''(\xi)\approx |a|^{\gamma-2}$, by which we mean that there exist two constants $0<c\leq C$ such that $-C|a|^{\gamma-2}\leq \omega_\gamma''(\xi)\leq -c |a|^{\gamma-2}$. This concludes the first case. 

When $a$ and $z$ are not colinear, we write $z=z_1+z_2$, where $z_1$ points in the direction of $a$ and $z_2$ is perpendicular to $ a$. By Taylor expansion in $z_2$,
\begin{align*}
    0\leq &\, \omega_\gamma(|a+z_1+z_2|)-\omega_\gamma(|a+z_1|)\\
    =&\,\nabla_z\omega_\gamma(|a+z_1|)\cdot z_2+\frac{1}{2}z_2\cdot D^2_z\omega_\gamma(|\xi|)z_2,
\end{align*}
where now $\xi$ is some point belonging to the line segment $[a+z_1,a+z_1+z_2]$. We can compute the Hessian to get
\[
    D^2_z\omega_\gamma(|\xi|)=\omega_\gamma''(|\xi|)\frac{\xi\otimes \xi}{|\xi|^2}+\frac{\omega_\gamma'(|\xi|)}{|\xi|}\left(I-\frac{\xi\otimes\xi}{|\xi|^2}\right),
\]
and hence
\[
    |D^2_z\omega_\gamma(|\xi|)|\leq C|\xi|^{\gamma-2}\leq C|a|^{\gamma-2}.
\]
Moreover, by radial symmetry, $\nabla_z \omega_\gamma(|a+z|)$ is perpendicular to $z_2$ at $z=z_1$. Thus, we proved
\[
    0\leq  \omega_\gamma(|a+z_1+z_2|)-\omega_\gamma(|a+z_1|)\leq C|a|^{\gamma-2}|z_2|^2.
\]
By the first part of this proof,
\[
    \omega_\gamma(|a|)+\omega_\gamma'(|a|)\frac{a}{|a|}\cdot z_1-\omega(|a+z_1|)\approx |a|^{\gamma-2}|z_1|^2.
\]
Therefore,
\begin{align*}
    \delta^2\omega_\gamma(|\cdot|)(a,z)=& \omega_\gamma(|a|)+\omega'_\gamma(|a|)\frac{a}{|a|}\cdot(z_1+z_2)-\omega_\gamma(|a+z_1+z_2|)\\
    =&\, \left[ \omega_\gamma(|a|)+\omega_\gamma'(|a|)\frac{a}{|a|}\cdot z_1-\omega_\gamma(|a+z_1|) \right]\\
    & -\big[ \omega_\gamma(|a+z_1+z_2|)-\omega_\gamma(|a+z_1|) \big]\\
    \leq&\, C|a|^{\gamma-2}|z_1|^2.
\end{align*}
We can also bound it from below by
\begin{align*}
    \delta^2\omega_\gamma(|\cdot|)(a,z)\geq &\, c_1|a|^{\gamma-2}|z_1|^2-c_2
|a|^{\gamma-2}|z_2|^2\\
\geq &\, c_1|a|^{\gamma-2}|z|^2,
\end{align*}
where we used the condition $z\in \C(a)$ in the last estimate. This concludes the proof.
\end{proof}

Once we have estimated the second increment of the modulus of continuity $\omega$, the next result follows easily.

\begin{lemma}\label{l:estimate-2o-incremental}
Let $\phi$ be as in \eqref{eq:phi-def}. Let $x,y \in B_{3/4}$ and assume that $|a| \coloneqq |x-y|$ is sufficiently small. Then, the following estimates hold for $z \in \C(a)$. 
\begin{itemize}
\item[(i)] For \( \omega(r) = \omega_\gamma(r)\), we have
\[
    c L |a|^{\gamma - 2} |z|^2 
    \;\le\;
    \delta^2 \phi(\cdot, y)( x,z)
    \;\le\;
    C L | a|^{\gamma - 2} |z|^2.
\]

\item[(ii)] For $\omega(r)=\widetilde \omega(r)$, we have
\[
    c L \frac{|z|^2}{|a| \log^2(|a|)}
    \;\le\;
    \delta^2 \phi(\cdot,y)( x,z)
    \;\le\;
    C L \frac{|z|^2}{|a| \log^2(| a|)}.
\]

\item[(iii)] The same estimates hold for \(\delta^2 \phi( x, \cdot)( y,z)\) in both cases above.

\end{itemize}
The constants $c$ and $C$ above are positive and universal.
\end{lemma}

\begin{proof}
The estimate for the case $\omega=\widetilde\omega$ is exactly the same as  \cite[Corollary 6.5]{GJS}. We only prove the estimate for $\omega=\omega_\gamma$. 

The second increment of $\omega_\gamma$ is given by Lemma \ref{l:2nd-increment-omega}. The function $\psi$ is smooth, hence its second increment is bounded by the maximum of its Hessian. So we get
\[
    |\delta^2 \psi(x,z)|\leq C|z|^2\leq \bar c|a|^{\gamma-2}|z|^2,
\]
where we can make $\bar c$ small by taking $|a|$ smaller. Therefore, we make it so small that we can absorb this term into the first one. This concludes the proof.
\end{proof}

In the above lemma, we obtain a positive growth for $\delta^2\phi$ for directions in $\C(a)$. For the remaining directions $z\in B_{|a|/2}$, the best estimate we can hope for is $|D^2\phi(x,y)|\leq C L \omega'(|a|)/|a|$. Thus, for both cases of $\omega$, if $|z|<|a|/2$, then we have the useful bound
\begin{align}\label{eq:gross_bound_phi}
    |\delta^2\phi(x,\cdot)(y,z)|\leq CL\frac{|\omega'(|a|)|}{|a|}|z|^2.
\end{align}

\vspace{0.1cm}

\subsection{Decomposition of the domain}~
In this section, we estimate the left-hand side of \eqref{eq:evaluate-at-max}, which we denoted by $\II$, using ideas from \cite{BS} (see also \cite{GJS}), by splitting the domain of integration into four parts and estimating each of them separately. 

Recall that we denoted $a_K=x_K-y_K$.
Let $\mathcal{C} = \mathcal{C}(a_K)$ be the cone of directions defined in \eqref{def:cone}, with parameter $\delta_0$. Let also
\[
    \mathcal{D}_1 = B_{\delta_1\,|a_K|} \cap \mathcal{C}^c,
    \qquad
    \mathcal{D}_2 = B_{1/16} \setminus (\mathcal{D}_1 \cup \mathcal{C}).
\]
Here $\delta_1$ is a small quantity to be chosen later, depending possibly on $|a_K|$. We will assume that $\delta_1\ll \delta_0$.

Given $D \subset \R^d$, we introduce the useful notation
\begin{equation*}
    \begin{aligned}
    &\LL[D]w(x)  \coloneqq \int_D J_p\bigl(w(x)-w(x+z)\bigr)\,|z|^{-d-sp}\dd z\\
    & = \int_D \bigl|w(x)-w(x+z)\bigr|^{p-2}\bigl(w(x)-w(x+z)\bigr)\,|z|^{-d-sp}\dd z.
\end{aligned}
\end{equation*}
We then write the quantity $\II$ from \eqref{eq:evaluate-at-max} as 
\begin{eqnarray*}
   \II &  = &\LL[\mathcal{C}]w_1(x_K,t_K)-\LL[\mathcal{C}]w_2(y_K,\tau_K)\\
    & &+ \LL[\mathcal{D}_1]w_1(x_K,t_K)-\LL[\mathcal{D}_1]w_2(y_K,\tau_K)\\
     & &+  \LL[\mathcal{D}_2]w_1(x_K,t_K)-\LL[\mathcal{D}_2]w_2(y_K,\tau_K)\\
     & &+ \LL[B_{1/16}^c]w_1(x_K,t_K)-\LL[B_{1/16}^c]w_2(y_K,\tau_K)\\
     & \eqqcolon &\II_1 + \II_2 + \II_3 + \II_4.
\end{eqnarray*}

We start with the estimate in the concavity cone.

\begin{lemma}[Estimate in $\C$]\label{l:estimate_I1}
For $L$ sufficiently large, we have
\begin{align}\label{eq:estimate_I1}
    \II_1\geq c(L\omega'(|a_K|))^{p-2}\int_{\C} |z|^{p-2-d-sp}\delta^2\phi(\cdot, y_K)(x_K,z)\dd z.
\end{align}
The constants $L$ and $c$ are universal.
\end{lemma}

\begin{proof}
Recall that 
\[
    \II_1=\LL[\C]w_1(x_K, t_K)-\LL[\C]w_2(y_K, \tau_K).
\]
We estimate only the term with $x_K$, since the other term can be bound identically.

Denote $\ell (z) \coloneqq \nabla_x \phi(x_K, y_K)\cdot z $, and note that since $\C$ is symmetric, we get
\[
    \LL[\C](\ell)=0.
\]
From \eqref{eq:IL_contradiction}, we obtain 
\[
    \Psi_K(x_K+z,y_K, t_K, \tau_K)\leq \Psi_K(x_K, y_K, t_K, \tau_K), \quad z\in \C.
\]
We note that the terms depending on time cancel perfectly, and we obtain
\[
    w_1(x_K+z, t_K)-w_1(x_K,t_K)\leq \phi(x_K+z, y_K)-\phi(x_K, y_K), \quad z\in \C.
\]
From the monotonicity of $J_p$ and applying Lemma \ref{l:Jp_est}, we get
\begin{align*}
\LL[\mathcal C]w_1(x_K, t_K)
&\ge \LL[\mathcal C]\phi(\cdot, y_K)(x_K) \\
&= \LL[\mathcal C]\phi(\cdot, y_K)(x_K)
   - \LL[\mathcal C]\ell(x_K) \\
&= (p-1)\int_{\mathcal C}\int_0^1 
   \bigl|\ell(z) + \eta \delta^2\phi(\cdot, y_K)(x_K,z)\bigr|^{p-2} \\
&\qquad\qquad \times \delta^2\phi(\cdot, y_K)(x_K,z)
   \, d\eta \, |z|^{-d-sp}\, dz .
\end{align*}
To further bound this integral, we start by noting that Lemma \ref{l:estimate-2o-incremental} implies that $\delta^2\phi(\cdot, y_K)(x_K,z)$ is strictly positive for $z\in \C$. Note also that 
\[
    \nabla_x \phi(x,y)=L\omega'(|x-y|)\frac{x-y}{|x-y|}+L_2\nabla \psi(x)
\]
and that, for both cases $\omega=\omega_\gamma$ and $\omega=\widetilde \omega$, we have that $\omega'(|x_K-y_K|)$ is away from $0$. Hence, by taking $L$ large, depending only on $L_2$, we get
\[
    |\ell(z)|\geq \frac{L}{2}\omega'(|a_K|) |z|, \quad z\in \C.
\]
Thus, we can further estimate
\[
    \LL[\C]w_1(x_K, t_K)\geq c\left(L\omega'(|a_K|)\right)^{p-2}\int_{\C}|z|^{p-2-d-sp}\delta^2 \phi(\cdot, y_K)(x_K, z)\dd z,
\]
as intended. Similarly, we get
\[
    \LL[\C]w_2(y_K, \tau_K)\leq -c\left(L\omega'(|a_K|)\right)^{p-2}\int_{\C}|z|^{p-2-d-sp}\delta^2 \phi(x_K, \cdot)(z, y_K)\dd z.
\]
This concludes the proof.
\end{proof}

Next, we write the estimate on $\mathcal{D}_1$. Recall that $\delta_1$ is the parameter appearing in the definition of $\mathcal{D}_1$.

\begin{lemma}[Estimate in $\mathcal{D}_1$]\label{l:estimate_I2}
Let $(x_K, y_K, t_K,\tau_K)$ satisfy \eqref{eq:IL_contradiction}. Then for $L$ sufficiently large, we have
\begin{align}\label{eq:estimate_I2}
    \II_2\geq -C (L\omega'(|a_K|))^{p-1} \delta_1^{p(1-s)}|a_K|^{p(1-s)-1}.
\end{align}
The constants $L$ and $C$ are universal.
\end{lemma}

\begin{proof}
As before, from \eqref{eq:IL_contradiction} and the monotonicity of $J_p$, we have
\begin{align*}
    &\LL[\mathcal{D}_1]w_1(x_K, t_K)\\
    \geq\,& (p-1)\int_{\mathcal{D}_1}\int_0^1|\ell(z)+\eta \delta^2\phi(\cdot,y_K)(x_K, z)|^{p-2}\delta^2\phi(\cdot,y_K)(x_K, z)\dd\eta |z|^{-d-sp}\dd z.
\end{align*}
We estimate $\delta^2\phi(\cdot,y_K)(x_K, z)$ using \eqref{eq:gross_bound_phi}. Moreover, by taking $L$ larger, depending on $L_2$ and universal constants, we obtain, for every $\eta \in [0,1]$,
\begin{align*}
    |\ell(z)+\eta\delta^2\phi(\cdot, y_K)(x_K, z)|\leq &\, CL|\omega'(|a_K|)|\, |z|+CL\frac{|\omega'(|a_K|)}{|a_K|}|z|^2\\
    \leq &\,CL|\omega'(|a_K|)|\, |z|.
\end{align*}
Therefore, we get
\begin{align*}
    \LL[\mathcal{D}_1]w_1(x_K, t_K)\geq&\, -C (L\omega'(|a_K|))^{p-1} |a_K|^{-1} \int_{\mathcal{D}_1} |z|^{p-d-sp\dd z}\\
    =&\,-C (L\omega'(|a_K|))^{p-1} \delta_1^{p(1-s)}|a_K|^{p(1-s)-1}.
\end{align*}
An identical bound holds for $-\LL[\mathcal{D}_1]w_2(y_K, \tau_K)$.
\end{proof}

We now bound the term in $\mathcal{D}_2$.

\begin{lemma}[Estimate in $\mathcal{D}_2$]\label{l:estimate_I3}
Suppose that $u$ satisfies \eqref{assumption:bat-ind} for some
\[
    \kappa\in(0,\min\{sp/(p-2),1\}).
\]
Let $K$ be large enough to satisfy \eqref{eq:InshSil}. Then, for any $\theta\in(0,1)$ and for any $L$ large enough universal, we have
\begin{equation}\label{eq:estimate_I3}
    \begin{aligned}
\II_3 \ge -C  \biggl[
&\int_{\delta_1|a_K|}^{|a_K|^{\theta}} r^{\kappa(p-2)+1-sp}\,\dd r
\allowbreak + |a_K|^{\frac{m-1}{m}\kappa}
\int_{\delta_1|a_K|}^{|a_K|^{\theta}} r^{\kappa(p-2)-sp}\,\dd r  \\
&\allowbreak + |a_K|^{\kappa+\theta(\kappa(p-2)-sp)}
\biggr].
\end{aligned}
\end{equation}
Here, $C$ depends on $\kappa, [u]_{C_x^{0,\kappa}}$ and universal constants.
\end{lemma}

\begin{proof}
Let $ \hat \delta \coloneqq |a_K|^{\theta} $ for some $\theta\in(0,1)$ fixed. Since $|a_K|$ is small, we have $\delta_1|a_K| < \hat \delta < \frac{1}{16}$. We divide $\II_3 $ in two parts as follows
\begin{align*}
    \II_3 =&\, \underbrace{\LL[\mathcal{D}_2 \cap B_{\hat \delta}] u(x_K, t_K) -\LL[\mathcal{D}_2  \cap B_{\hat \delta}] u(y_K, \tau_K)}_{\II_{3,1}} \\
    &+\,\underbrace{\LL[\mathcal{D}_2  \cap B_{\hat \delta}^c] u(x_K,  t_K) -\LL[\mathcal{D}_2  \cap B_{\hat \delta}^c] u(y_K, \tau_K)}_{\II_{3,2}}.
\end{align*}
Define $\delta^1 f(x,z) \coloneqq f(x)- f(x+z)$. As before, we write $\II_{3,1}$ as
\begin{align*}
    \II_{3,1} &= (p-1) \int_{E_{\hat \delta}} \int_{0}^1 \Big|\delta^1 u(\cdot, \tau_K)(y_K,z) +\\
    & \quad +\eta(\delta^1 u(\cdot, t_K)(x_K,z) -\delta^1 u(\cdot, \tau_K)(y_K,z))\Big|^{p-2} \times \\ 
    &  \qquad  \times \big(\delta^1 u(\cdot, t_K)(x_K,z) -\delta^1 u(\cdot,\tau_K)(y_K,z) \big)  \dd  \eta |z|^{-d-sp}  \dd  z,
\end{align*}
where $E_{\hat \delta} \coloneqq \mathcal{D}_2 \cap B_{\hat \delta}$. Recalling that $(x_K, t_K)$ and $(y_K, \tau_K)$ are the points where the maximum in \eqref{eq:IL_contradiction} is attained, we see that $\Psi_K(x_K , y_K, t_K, \tau_K) \geq \Psi_K(x_K + z, y_K + z, t_K, \tau_K)$ and
\[
    \delta^1 u(\cdot, t_K)(x_K,z) -\delta^1 u(\cdot, \tau_K)(y_K,z) \geq C L_2 \delta^1 \psi(x_K,z),
\]
which implies 
\begin{align*}
    \II_{3,1}&\, \geq   C L_2 \int_{E_{\hat \delta}} \int_{0}^1 \Big|\delta^1 u(\cdot, \tau_K)(y_K,z) +\\
    & \quad +\eta(\delta^1 u(\cdot, t_K)(x_K,z) -\delta^1 u(\cdot, \tau_K)(y_K,z) )\Big|^{p-2}\times\\
    &\qquad\times\,\delta^1 \psi(x_K,z)  \dd  \eta  |z|^{-d-sp}  \dd  z.
\end{align*}
Since $u\in C^{0,\kappa}(Q_{5/6})$ in space, we get 
\begin{align*}
    |\delta^1 (\cdot, \tau_K)(y_K,z) + \eta(\delta^1 &u(\cdot, t_K)(x_K,z) -\delta^1 u(\cdot, \tau_K)(y_K,z))|\\
    &\leq 3[u]_{C^{0,\kappa}(Q_{5/6})}|z|^\kappa.
\end{align*}
Using Taylor's expansion of $\psi$, we get
\[
    |\delta^1\psi(x_K, z)|\leq C\left(|z|^2+|\nabla \psi(x_K)|\,|z|\right).
\]
We can further bound $\II_{3,1}$ by
\begin{align*}
    \II_{3,1} &\geq - C L_2 \bigg[\int_{\mathcal{D}_2\cap B_{\hat \delta}} |z|^{\kappa(p-2)+2-d-sp}  \dd  z
    \\
    & \quad + |\nabla \psi(x_K)| \int_{\mathcal{D}_2 \cap B_{\hat \delta}} |z|^{\kappa(p-2)+1-d-sp}  \dd  z   \bigg]\\
    &\geq  - CL_2 \left[\int_{\delta_1|a_K|}^{|a_K|^{\theta}} r^{\kappa(p-2)+1-sp}  \dd  r
    + |\nabla \psi(x_K)| \int_{\delta_1|a_K|}^{|a_K|^{\theta}}  r^{\kappa(p-2)-sp}  \dd  r  \right].
\end{align*}
From \eqref{eq:IL_contradiction} and \eqref{eq:InshSil}, we also have
\begin{align*}
    \psi(x_K) & <\frac{1}{L_2}(u(x_K, t_K)-u(y_K, \tau_K))\\
    & \leq \frac{1}{L_2}(C_1(\kappa)|a_K|^\kappa + C_2|b_K|^{\kappa_0}) \leq \frac{C(\kappa)}{L_2}|a_K|^\kappa.
\end{align*}
Recalling \eqref{eq:inequality_psi}, we get
\[
    |\nabla \psi(x_K)|\leq \left(C(\kappa)\frac{|a_K|^\kappa}{L_2}\right)^\frac{m-1}{m} \eqqcolon C_1 |a_K|^{\frac{m-1}{m}\kappa}.
\]
We use this to conclude the estimate for $\II_{3,1}$
\[
    \II_{3,1}\geq -C L_2 \left[\int_{\delta_1|a_K|}^{|a_K|^{\theta}} r^{\kappa(p-2)+1-sp}  \dd  r
    + C_1 |a_K|^{\frac{m-1}{m}\kappa} \int_{\delta_1|a_K|}^{|a_K|^{\theta}}  r^{\kappa(p-2)-sp}  \dd  r  \right].
\]
For the term $\II_{3,2}$, we can proceed in a similar way to get

\begin{align*}
    \II_{3,2} &\geq - C (|a_K|^\kappa+|b_K|^{\kappa_0})\int_{\mathcal{D}_2 \cap B_{\hat \delta}^c} |z|^{\kappa(p-2)-d-sp}  \dd  z \\
    &\geq -  C|a_K|^\kappa \int_{\hat \delta}^{\frac{1}{16}} r^{\kappa(p-2)-sp-1}  \dd  r
    \\
    &\geq -C |a_K|^{\kappa+\theta(\kappa(p-2)-sp)}. 
\end{align*}

The proof follows by combining both estimates.
\end{proof}

Finally, we control the term involving the tail. Here, we use Assumption \eqref{eq:InshSil} in a crucial way. 

\begin{lemma}[Estimate on $B_{1/16}^c$]\label{l:estimate_I4}
Suppose $u$ satisfies Assumption \eqref{assumption:bat-ind} for some $\kappa\in(0,1]$ and let $K$ be large enough to satisfy \eqref{eq:InshSil}. For $L$ large enough, we have
\begin{align}\label{eq:estimate_I4}
    \II_4\geq -C |a_K|^\kappa,
\end{align}
where  $C$ depends on universal constants and $[u]_{C^{\kappa}_x}$.
\end{lemma}

\begin{proof}
Recall that
\begin{align*}
    |\II_4|&= \Bigg| \int_{B_{\frac{1}{16}}^c(x_K)} J_p(u(x_K, t_K)-u(z, t_K))|x_K-z|^{-d-sp} \dd z \\
    &\qquad -\int_{B_{\frac{1}{16}}^c(y_K)} J_p(u(y_K, \tau_K)-u(z, \tau_K))|y_K-z|^{-d-sp} \dd z \Bigg|.
\end{align*}
We break $\II_4$ into three terms $|\II_4|\leq \mathcal{J}_1+\mathcal{J}_2+\mathcal{J}_3,$
where
\begin{align*}
    \mathcal{J}_1\!\coloneqq &\left|\int_{B^c_{\frac{1}{16}}\!\!(x_K)}J_p(u(x_K, t_K)-u(z, t_K))\!\left( |x_K-z|^{-d-sp}\!-|y_K-z|^{-d-sp}\right)\!\dd z\right|,\\
    \mathcal{J}_2\!\coloneqq &\,\Bigg|\int_{B^c_{\frac{1}{16}}( x_K)}J_p(u(x_K, t_K)-u(z, t_K))|y_K-z|^{-d-sp}\dd z\\
    & \qquad -\int_{B^c_{\frac{1}{16}}(y_K)}J_p(u(x_K, t_K)-u(z, t_K))|y_K-z|^{-d-sp}\dd z\Bigg|
\end{align*}
and
\begin{align*}
    \mathcal{J}_3\coloneqq \left|\int_{B^c_{\frac{1}{16}}(y_K)}\frac{J_p(u(x_K, t_K)-u(z, t_K))-J_p(u(y_K, \tau_K)-u(z, \tau_K))}{|y_K-z|^{d+sp}}\dd z \right|
\end{align*}

When $|x_K-z| > \frac{1}{16}$ and $|a_K|$ is small, for some constant $C$, we have
\[
    \left||x_K-z|^{-d-sp}-|y_K-z|^{-d-sp}\right| \leq C|x_K-z|^{-d-sp-1}|a_K|.
\]
We bound the first term as
\[
    \mathcal{J}_1 \leq C|a_K|.
\]
To bound $\mathcal{J}_2$, we can assume that $|x_K-y_K|<1/32$ up to taking $K$ large. Note that we need only to integrate in the symmetric difference of the balls $B_\frac{1}{16}(x_K)$ and $B_\frac{1}{16}(y_K)$, which we denote by $B_\frac{1}{16}(x_K) \,\Delta \, B_\frac{1}{16}(y_K)$. Since $|B_\frac{1}{16}(x_K)\, \Delta \, B_\frac{1}{16}(y_K)| \leq C|x_K-y_K|$ and for $z \in B_\frac{1}{16}(x_K) \Delta B_{\frac{1}{16}}(y_K)$, we have $|y_K-z| \geq \frac{1}{16}-|a_K| \geq \frac{1}{32}$, we estimate $\mathcal{J}_2$ as 
\[
    \mathcal{J}_2 \leq  C|a_K|.
\]
To estimate $\mathcal{J}_3$ we use Lemma \ref{l:Jp_est} to bound $J_p(u(x_K, t_K)-u(z, t_K))-J_p(u(y_K, \tau_K)-u(z, \tau_K))$. More precisely, by H\"older continuity of $u$ in space and time, we obtain
\begin{align*}
    &|u(x_K,t_K)-u(z,t_K) - (u(y_K,\tau_K)-u(z,\tau_K))|\\
    \leq &\, C|a_K|^\kappa + |u(z,t_K)-u(z,\tau_K)| +C|b_K|^{\kappa_0},
\end{align*}
which, using that $|b_K|^{\kappa_0}\leq |a_K|^\kappa$, allows us to get 
\begin{align*}
    \mathcal{J}_3 \leq C|a_K|^\kappa +  C\int_{B_\frac{1}{16}^c} &\left( \|u\|_{L^{\infty}(Q_1)}+|u(y_K+z,t_K)| +|u(y_K+z,\tau_K)|\right)^{p-2}\times\\
    &\times|u(y_K+z,t_K)-u(y_K+z,\tau_K)| |z|^{-d-sp}\dd z.
\end{align*}
Using H\"older inequality with exponents $p-1$ and $\frac{p-1}{p-2}$, we obtain
\[
    \mathcal{J}_3\leq C  \left[ |a_K|^\kappa + \left(\int_{B_{\frac{1}{16}}^c} |u(y_K+z, t_K)-u(y_K+z, \tau_K)|^{p-1}|z|^{-d-sp} \dd z \right)^{\frac{1}{p-1}}\right].
\]
By Assumption \eqref{eq:InshSil}, we conclude that
\[
    \mathcal{J}_3\leq C  |a_K|^\kappa .
\]
The proof is complete.
\end{proof}

\subsection{Proof of the Lipschitz regularity in space}\label{ss:ael_tomuma_noiist}

We gather the previous results and prove the Lipschitz regularity in the spatial variable. 

\begin{theorem}[Lipschitz in space]\label{t:lip-in-space_Chibrata}
Let $u$ be a viscosity solution to \eqref{eq:parabolic-fractional-p-laplacian} in $Q_1$. Suppose further that $u \in C_{\rm loc}(-1,0;L^{p-1}_{sp}(\R^d))$ and \eqref{assumption:normalized-setting} and \eqref{assumption:bat-ind} are in force. Then, $u$ is Lipschitz in space at $t_0$ and satisfies
\[
    |u(x,t_0)-u(y,t_0)|\leq C|x-y|, \quad \text{for all }  x,y \in B_{1/2}.
\]
The constant $C$ is universal.
\end{theorem}

\begin{proof}
Denote $\bar \gamma \coloneqq \min\{1,sp/(p-1)\}$. For each fixed $L$, we take $K$ large enough to satisfy \eqref{eq:InshSil}.

We split the proof into the following five steps, assuming that $t_0=0$ in each step. 
\begin{enumerate}
    \item We show that $u\in C^{0,\bar \gamma^-}_x$ locally,
    \item We show that $u\in C^{0,\bar \gamma}_x$ locally when $\bar \gamma<1$,
    \item We show that $u\in C^{0,1}_x$ locally when $\bar \gamma=1$,
    \item We show that $u\in C^{0,1^-}_x$ locally when $\bar \gamma<1$,
    \item We show that $u\in C^{0,1}_x$ locally when $\bar \gamma<1$.
\end{enumerate}

\noindent \textbf{Step 1:} We use a bootstrapping argument to prove that $u\in C^{0,\gamma}_x$ locally for every $\gamma<\bar\gamma$. Suppose that $u\in C^{0,\kappa}_x$ locally for some $\kappa\in(0,\gamma)$. Set
\[
    \kappa_1
    =
    \min\left\{
        \gamma,\,
        \kappa+\frac{1}{2(p-1)},\,
        \kappa+\frac{sp-(p-2)\gamma}{2(p-1)}
    \right\}.
\]
Since
\[
    \gamma<\frac{sp}{p-1}<\frac{sp}{p-2},
\]
we have $sp-(p-2)\gamma>0$. Moreover, $\kappa_1<\kappa+1/(p-1)$, and therefore we can choose $m\geq 3$ large enough so that
\begin{align}\label{eq:m_choice_step1}
    \kappa(p-2)+1
    \geq
    \kappa_1(p-1)-\frac{m-1}{m}\kappa.
\end{align}
In this step, we take $\delta_0$ equal to the universal constant $\bar c_0$ in Lemma \ref{l:2nd-increment-omega}, and we choose $\omega$ to be the H\"older profile $\omega_{\kappa_1}$. Define
\[
    \mathcal{E} \coloneqq \kappa_1(p-1)-sp.
\]
Since $\kappa_1<sp/(p-1)$, we have
\[
    \mathcal{E} < 0.
\]
By Lemma \ref{l:estimate_time}, we have
\begin{align}\label{eq:time_again1}
    \II_1+\II_2+\II_3+\II_4
    \leq
    C |a_K|^{\frac{\kappa \beta^\ast}{1+\beta^\ast}}
    \leq C_1.
\end{align}
We estimate each of these four terms. By Lemma \ref{l:estimate_I1}, we have
\[
    \II_1\geq cL^{p-1}|a_K|^{\mathcal{E}}.
\]
For $\II_2$, Lemma \ref{l:estimate_I2} gives
\[
    \II_2\geq
    -CL^{p-1}|a_K|^{\mathcal{E}}\delta_1^{p(1-s)}.
\]
We take $\delta_1$ universally small, depending on $\delta_0$, so that
\begin{align}\label{eq:fix_delta}
    \II_1+\II_2
    \geq
    c_1L^{p-1}|a_K|^{\mathcal{E}}.
\end{align}
We now choose $\theta$. Since
\[
    \kappa_1
    <
    \kappa+\frac{sp-(p-2)\gamma}{p-1}
    <
    \kappa+\frac{sp-(p-2)\kappa}{p-1}
    =
    \frac{sp+\kappa}{p-1},
\]
we obtain
\[
    \mathcal{E} = \kappa_1(p-1)-sp<\kappa.
\]
Also,
\[
    sp-\kappa(p-2)>0,
\]
because $\kappa<\gamma<sp/(p-2)$. Hence we can choose $\theta\in(0,1)$ so small that
\begin{align}\label{eq:theta_choice_step1}
    \kappa+\theta(\kappa(p-2)-sp)>\mathcal{E}.
\end{align}

We next estimate $\II_3$ by applying Lemma \ref{l:estimate_I3} with this value of $\theta$. The third term is directly controlled by \eqref{eq:theta_choice_step1}:
\[
    |a_K|^{\kappa+\theta(\kappa(p-2)-sp)}
    \leq
    |a_K|^{\mathcal{E}},
\]
for $|a_K|<1$. For the second term, using \eqref{eq:m_choice_step1}, we have
\[
    \kappa(p-2)-sp
    \geq
    \mathcal{E}-\frac{m-1}{m}\kappa-1.
\]
Therefore, since $0<r<1$,
\begin{align*}
    |a_K|^{\frac{m-1}{m}\kappa}
    \int_{\delta_1|a_K|}^{|a_K|^\theta}
    r^{\kappa(p-2)-sp}\dd r
    &\leq
    |a_K|^{\frac{m-1}{m}\kappa}
    \int_{\delta_1|a_K|}^{1}
    r^{\mathcal{E}-\frac{m-1}{m}\kappa-1}\dd r   \\
    &\leq
    C_{\delta_1}
    |a_K|^\mathcal{E}.
\end{align*}
Recall $\delta_1$ was chosen universally small. For the first term, we use the choice
\[
    \kappa_1\leq \kappa+\frac{1}{2(p-1)}.
\]
This gives
\[
    \mathcal{E} = \kappa_1(p-1)-sp
    \leq
    \kappa(p-1)+\frac12-sp.
\]
Hence
\[
    \mathcal{E}-1
    \leq
    \kappa(p-1)-sp-\frac12
    <
    \kappa(p-2)+1-sp.
\]
Thus, again because $0<r<1$,
\begin{align*}
   \int_{\delta_1|a_K|}^{|a_K|^\theta}
   r^{\kappa(p-2)+1-sp}\dd r
   &\leq
   \int_{\delta_1|a_K|}^{1}
   r^{\mathcal{E}-1}\dd r  \\
   &\leq
   C_{\delta_1}|a_K|^\mathcal{E}.
\end{align*}
Combining the previous three estimates, we obtain
\[
    \II_3\geq -C_2|a_K|^\mathcal{E},
\]
where $C_2$ may depend on the fixed universal parameter $\delta_1$.

Finally, Lemma \ref{l:estimate_I4} gives
\[
    \II_4\geq -C|a_K|^\kappa\geq -C_3.
\]
Combining these estimates with \eqref{eq:time_again1}, we reach
\[
    |a_K|^\mathcal{E}\left(c_1L^{p-1}-C_2\right)
    \leq
    C_1+C_3.
\]
For $L$ sufficiently large, the factor $c_1L^{p-1}-C_2$ is positive and tends to infinity. Since $\mathcal{E}<0$ and $|a_K|<1$, we have $|a_K|^\mathcal{E}\geq 1$. Therefore, the left-hand side becomes arbitrarily large as $L\to\infty$, contradicting the previous inequality. This proves that $u\in C_x^{0,\kappa_1}$ locally.

If $\kappa_1=\gamma$, we are done. Otherwise, we repeat the argument with $\kappa$ replaced by $\kappa_1$. Since, until the last step,
\[
    \kappa_{j+1}-\kappa_j
    \geq
    \min\left\{
        \frac{1}{2(p-1)},\,
        \frac{sp-(p-2)\gamma}{2(p-1)}
    \right\}>0,
\]
the iteration reaches $\gamma$ after finitely many steps. Hence $u\in C^{0,\gamma}_x$ locally. Since $\gamma<\bar\gamma$ was arbitrary, we conclude that $u\in C^{0,\bar\gamma^-}_x$ locally.

\vspace{0.1in}

\noindent \textbf{Step 2: } We show that $u\in C^{0,\bar \gamma}_x$ locally, when
\[
    \bar \gamma = \frac{sp}{p-1} < 1.
\]
From Step 1, we know that $u \in C^{0,\kappa}_x(Q_1)$ for all $\kappa<\bar \gamma$. 

In this step, we take $\delta_0$  equal to the universal constant $\bar c_0$ in Lemma \ref{l:2nd-increment-omega} and $\omega$ to be the H\"older profile $\omega_{\bar \gamma}$. We use Lemma \ref{l:estimate_time} to get
\begin{align}\label{eq:time_again}
    \II_1 + \II_2 + \II_3 + \II_4\leq C |a_K|^\frac{\kappa \beta^\ast}{1+\beta^\ast}\leq C_0.
\end{align}
We estimate each of these four terms. By Lemmas \ref{l:estimate_I1} and \ref{l:estimate-2o-incremental}, we have, since $\bar \gamma (p-1) -sp = 0$,
\begin{align*}
    \II_1 \geq cL^{p-1}|a_K|^{\bar \gamma(p-1)-sp}\delta_0^{d-1}=cL^{p-1}\delta_0^{d-1}.
\end{align*}
For $\II_2$, we use Lemma \ref{l:estimate_I2} to get
\begin{align*}
    \II_2 \geq -CL^{p-1}|a_K|^{\bar \gamma(p-1)-sp}\delta_1^{p(1-s)}=-CL^{p-1}\delta_1^{p(1-s)}.
\end{align*}
We can take $\delta_1$ to be universally small, depending also on $\delta_0$, so that
\begin{align*}
    \II_1 + \II_2 \geq c_1L^{p-1}.
\end{align*}

Now we deal with the term $\II_3$. We choose $\kappa<\bar \gamma $ large so that
\[
\kappa(p-2)-sp>-1.
\]
Then we take $\theta\in(0,1)$ in Lemma \ref{l:estimate_I3} small enough to satisfy
\[
    \kappa+\theta[\kappa(p-2)-sp]>0.
\]
Since $u\in C^{0,\kappa}_x$ locally, we get from \eqref{eq:estimate_I3} that
\[
    \II_3\geq -C_1,
\]
since all the exponents of $|a_K|$ in the three terms obtained are positive.

From Lemma \ref{l:estimate_I4}, we get
\[
    \II_4\geq -C|a_K|^\kappa\geq -C_2.
\]
Combining these estimates in \eqref{eq:time_again},
\[
    c_1L^{p-1}\leq C_1+C_2+C_3.
\]
Hence, the contradiction follows easily by taking $L$ large enough.

\vspace{0.1cm}

\noindent \textbf{Step 3: }We now consider the case 
\[
    \bar \gamma = 1 \leq \frac{sp}{p-1}.
\]
In this case, we take $\delta_0=\bar c_0|\log(|a_K|)|^{-1}$ and $\delta_1=\bar c_1|\log(|a_K|)|^{-q}$, where $\bar c_0$ is given by Lemma \ref{l:2nd-increment-omega}, $\bar c_1$ is a small number and $q$ is a large number to be fixed universally. We also take $\omega = \widetilde\omega$ to be the Lipschitz profile. 

Then, from Lemma \ref{l:estimate_I1} combined with the bounds from Lemma \ref{l:2nd-increment-omega}, we get
\begin{align*}
    \II_1&\geq cL^{p-2}\int_{\C} |z|^{p-2-d-sp}\left(L(|a_K|\log^2(|a_K|))^{-1}|z|^2\right)\dd z\\
    &\geq cL^{p-1}(|a_K|\log^2(|a_K|))^{-1}\delta_0^{d-1}\int_0^{|a_K|/2}r^{p(1-s)-1}\dd r\\
    &=cL^{p-1}(|a_K|\log^2(|a_K|))^{-1}\delta_0^{d-1}|a_K|^{p(1-s)}\\
    &=cL^{p-1}\delta_0^{d+1}|a_K|^{p(1-s)-1}.
\end{align*}
For $\II_2$, we get from Lemma \ref{l:estimate_I2} that
\[
    \II_2\geq -CL^{p-1}\delta_1^{p(1-s)}|a_K|^{p(1-s)-1}.
\]
Since $|a_K|$ is small, we can take $q$ large, universal, so that
\[
    c\delta_0^{d+1}-C\delta_1^{p(1-s)}>\frac{c}{2}\delta_0^{d+1},
\]
thus getting
\begin{align*}
    \II_1+\II_2\geq c_1L^{p-1}\delta_0^{d+1}|a_K|^{p(1-s)-1}.
\end{align*}
Before bounding $\II_3$, we choose the parameters $\kappa, m, \theta$ as follows. From Step 1, we can suppose $u\in C^{0,\kappa}_x(Q_1)$, for $\kappa\in(0,1)$ chosen large enough so that
\[
    \kappa(p-1)+1-sp>0.
\]
We then take $m$ large enough to satisfy
\[
    \frac{m-1}{m}\kappa +\kappa(p-2)+1-sp>0.
\]
Finally, we pick $\theta\in(0,1)$ small to satisfy
\[
    \kappa+\theta[\kappa(p-2)-sp]>0.
\]
With these choices, we bound each of the terms in \eqref{eq:estimate_I3}. By the choice of $\kappa$,  
\[
    \int_{\delta_1|a_K|}^{|a_K|^\theta} r^{\kappa(p-2)+1-sp}\dd r\leq C|a_K|^{\theta[\kappa(p-2)+2-sp]}.
\]
Since $p-1\leq sp$, we have $\kappa(p-2)-sp<-1$. Thus
\begin{align*}
        |a_K|^{\frac{m-1}{m}\kappa}\int_{\delta_1|a_K|}^{|a_K|^\theta} r^{\kappa(p-2)-sp}\dd r\leq&\, C|a_K|^{\frac{m-1}{m}\kappa}\delta_1^{\kappa(p-2)-sp+1}|a_K|^{\kappa(p-2)-sp+1}\\
        =&\,C\delta_1^{\kappa(p-2)-sp+1}|a_K|^{\frac{m-1}{m}\kappa+\kappa(p-2)-sp+1}.
\end{align*}
Hence, we get
\begin{align*}
\II_3 \geq&\, -C\Big[
    |a_K|^{\theta[\kappa(p-2)+2-sp]}
    +\delta_1^{\kappa(p-2)-sp+1}
        |a_K|^{\frac{m-1}{m}\kappa+\kappa(p-2)-sp+1}  \\
&\qquad\qquad
    +|a_K|^{\kappa+\theta[\kappa(p-2)-sp]}
\Big] \\
\geq&\,-C_1\delta_1^{\kappa(p-2)-sp+1}|a_K|^{\theta_\kappa},
\end{align*}
for some $\theta_\kappa>0$.

Finally, we use Lemma \ref{l:estimate_I4} to get
\[
    \II_4\geq -C_2|a_K|^\kappa.
\]
Using the estimates above for each term and \eqref{eq:time_again}, we obtain
\begin{align*}
        c_1L^{p-1}\delta_0^{d+1}|a_K|^{p(1-s)-1}\leq C_1\delta_1^{\kappa(p-2)-sp+1}|a_K|^{\theta_\kappa}+C_2|a_K|^\kappa + C |a_K|^\frac{\kappa \beta^\ast}{1+\beta^\ast}.
\end{align*}
Since $p(1-s)-1\leq 0$, this inequality can not hold for $L$ large enough (which implies $|a_K|$ small).

\vspace{0.1in}

\noindent \textbf{Step 4: }We prove that $u\in C^{0,\gamma}_x$ locally for every $\gamma<1$, when $sp<p-1$. Let $\delta_0$ be equal to $\bar c_0$ given by Lemma \ref{l:2nd-increment-omega}. We can also assume, from Step 2, that $u\in C_x^{0,\kappa}$ locally for
\[
    \kappa=\frac{sp}{p-1}<1.
\]
We choose
\[
    \kappa_1
    =
    \min\left\{
        \gamma,\,
        \kappa+\frac{\kappa(2-\kappa)}{4(p-1)}
    \right\}.
\]
Equivalently, since $\kappa=sp/(p-1)$,
\[
    \kappa_1
    =
    \min\left\{
        \gamma,\,
        \kappa+\frac{sp\,[2(p-1)-sp]}{4(p-1)^3}
    \right\}.
\]
Set
\[
    \mathcal{E} \coloneqq \kappa_1(p-1)-sp=(p-1)(\kappa_1-\kappa).
\]
By our choice of $\kappa_1$, we have
\[
    0<\mathcal{E}\leq \frac{\kappa(2-\kappa)}{4}.
\]
We take $m\geq 3$ so large that
\[
    \mathcal{E}<\frac{m-1}{m}\kappa.
\]
We start by proving that $u\in C^{0,\kappa_1}_x$ locally and then conclude using a bootstrapping argument. Here, we choose $\omega$ to be the H\"older modulus of continuity $\omega_{\kappa_1}$.

Using Lemmas \ref{l:2nd-increment-omega}, \ref{l:estimate_I1} and \ref{l:estimate_I2}, by taking $\delta_1$ universally small, we get
\[
    \II_1+\II_2\geq \frac{c}{2}L^{p-1}|a_K|^{\mathcal{E}}.
\]
Since $sp<p-1$, from Lemma \ref{l:estimate_I3} we obtain
\begin{align*}
        \II_3\geq
        -C\left[
        |a_K|^{\theta(2-\kappa)}
        + |a_K|^{\frac{m-1}{m}\kappa+\theta(1-\kappa)}
        + |a_K|^{\kappa(1-\theta)}
        \right],
\end{align*}
for every $\theta\in(0,1)$. We now choose
\[
    \theta=\frac{\mathcal{E}}{2-\kappa}
    =
    \frac{\kappa_1(p-1)-sp}{2-\kappa}.
\]
By the choice of $\kappa_1$, we have
\[
    \theta
    \leq \frac{\kappa}{4}
    <1.
\]
The third exponent is also strictly larger than $\mathcal{E}$. Indeed,
\[
    \kappa(1-\theta)
    \geq \kappa\left(1-\frac{\kappa}{4}\right),
\]
while
\[
    \mathcal{E} \leq \frac{\kappa(2-\kappa)}{4}.
\]
Hence
\[
    \kappa(1-\theta)- \mathcal{E}
    \geq
    \kappa\left(1-\frac{\kappa}{4}\right)
    -\frac{\kappa(2-\kappa)}{4}
    =
    \frac{\kappa}{2}
    >0.
\]
Finally, for the middle exponent, our choice of $m$ gives
\[
    \frac{m-1}{m}\kappa+\theta(1-\kappa)
    >
    \frac{m-1}{m}\kappa
    >
    \mathcal{E}.
\]
Consequently,
\[
    \II_3\geq -C_1|a_K|^{\mathcal{E}}.
\]
Finally, Lemma \ref{l:estimate_I4} gives
\[
    \II_4\geq -C_2 |a_K|^{\kappa},
\]
and Lemma \ref{l:estimate_time} gives
\[
    \II_1+\II_2+\II_3+\II_4\leq C_3|a_K|^{\frac{\kappa \beta^\ast}{1+\beta^\ast}}.
\]
Combining the estimates above, we arrive at
\begin{align}\label{eq:I123}
    \frac{c_1}{2}L^{p-1}|a_K|^\mathcal{E}
    -C_1|a_K|^\mathcal{E}
    -C_2|a_K|^\kappa
    \leq
    C_3|a_K|^{\frac{\kappa \beta^\ast}{1+\beta^\ast}}.
\end{align}
Since
\[
    \mathcal{E}\leq \frac{\kappa(2-\kappa)}{4}<\kappa,
\]
we can choose $\beta^\ast$ large enough so that
\[
    \kappa\frac{\beta^\ast}{1+\beta^\ast}> \mathcal{E}.
\]
Dividing \eqref{eq:I123} by $|a_K|^\mathcal{E}$, we obtain
\[
    \frac{c_1}{2}L^{p-1}
    -C_1
    -C_2|a_K|^{\kappa-\mathcal{E}}
    \leq
    C_3|a_K|^{\frac{\kappa \beta^\ast}{1+\beta^\ast}-\mathcal{E}}.
\]
Since $L\omega_{\kappa_1}(|a_K|)\leq 2$, we have $|a_K|\to0$ as $L\to\infty$. Therefore, taking $L$ sufficiently large yields a contradiction. This proves that $u\in C_x^{0,\kappa_1}$ locally.

If $\kappa_1=\gamma$, we are done. Otherwise, since the increment
\[
    \frac{\kappa(2-\kappa)}{4(p-1)}
\]
is strictly positive, we can repeat the argument finitely many times to reach any prescribed exponent $\gamma<1$. Hence $u\in C^{0,\gamma}_x$ locally for every $\gamma<1$.

\vspace{0.1in}

\noindent \textbf{Step 5: } We finally prove that $u\in C_x^{0,1}$ when $sp<p-1$. We take $\omega$ as the Lipschitz profile $\widetilde \omega$.

From Step 4, we know that $u\in C_x^{0,\kappa}$ locally for all $\kappa<1$. We take $\kappa<1$ satisfying
\begin{align}\label{eq:cond_kappa}
    \kappa>\max\left\{\frac{sp}{p-1}, p(1-s)-1\right\},
\end{align}
and $m$ so large that 
\[
    p(1-s)-1 < \frac{m-1}{m}\kappa.
\]
As in Step 3, we take
\[
    \delta_0 = \bar c_0|\log(|a_K|)|^{-1} \quad \text{and} \quad \delta_1 = \bar c_1|\log(|a_K|)|^{-q},
\]
where $\bar c_0$ is given by Lemma \ref{l:2nd-increment-omega}, $\bar c_1$ is small, and $q$ is large, to be chosen. Then, from Lemmas \ref{l:estimate_I1} and \ref{l:2nd-increment-omega}, we obtain
\begin{align*}
    \II_1&\geq cL^{p-1}\left(|a_K| \log^2(|a_K|)\right)^{-1}\delta_0^{d-1}\int_0^{\frac{|a_K|}{2}} r^{p(1-s)-1}\dd r\\
    &=cL^{p-1}\delta_0^{d+1}|a_K|^{p(1-s)-1}.
\end{align*}
For $\II_2$, we get from Lemma \ref{l:estimate_I2} that
\[
    \II_2\geq -CL^{p-1}\delta_1^{p(1-s)}|a_K|^{p(1-s)-1}.
\]
Since $|a_K|$ is small, we can take $q$ satisfying $d+1<qp(1-s)$, which implies
\[
    c\delta_0^{d+1}-C\delta_1^{p(1-s)}>\frac{c}{2}\delta_0^{d+1}.
\]
We arrive at
\[
    \II_1+\II_2\geq c_1L^{p-1}\delta_0^{d+1}|a_K|^{p(1-s)-1}.
\]
Now we bound $\II_3$, taking
\[
    \theta=\frac{p(1-s)-1+\varepsilon_1}{\kappa(p-2)+2-sp}.
\]
By our choice of $\kappa$, we have $\theta\in(0,1)$, for $\varepsilon_1$ small enough. Indeed, we notice that, since $p < 2/(1-s)$ and $sp/(p-1) < \kappa < 1$, we obtain
\[
    p(1-s) - 1 < 1,
\]
and 
\[
    \kappa(p-2) + 2 - sp \geq 2-\kappa > 1.
\]
Using this $\theta$ in Lemma \ref{l:estimate_I3}, we get
\[
    \int_{\delta_1|a_K|}^{|a_K|^\theta} r^{\kappa(p-2)+1-sp}\dd r\leq |a_K|^{p(1-s)-1+\varepsilon_1},
\]
using that $\kappa(p-2)+2-sp > 1$. Up to taking $m$ even larger, satisfying in addition that $\kappa/m<\theta$, we obtain
\[
    \kappa\frac{m-1}{m} +\theta(\kappa(p-2)-sp+1)>\kappa +\theta(\kappa(p-2)-sp).
\]
Therefore, we can write the second term in $\II_3$ as
\begin{align*}
    |a_K|^{\kappa\frac{m-1}{m}}\int_{\delta_1|a_K|}^{|a_K|^\theta} r^{\kappa(p-2)-sp}\dd r&\leq C|a_K|^{\kappa \frac{m-1}{m}+\theta(\kappa(p-2)-sp+1)}\\
    &\leq C|a_K|^{\kappa+\theta(\kappa(p-2)-sp)}.
\end{align*}
We finally analyze the term $|a_K|^{\kappa+\theta(\kappa(p-2)-sp)}$. We want to fix $\kappa<1$ large enough so that
\begin{align}\label{eq:kappa}
    \kappa+\theta(\kappa(p-2)-sp)\geq p-1-sp+\varepsilon_2,
\end{align}
for some $\varepsilon_2>0$. Since $\theta$ depends on $\kappa$, we define the auxiliary function
\[
    f(\kappa)=\kappa+\theta(\kappa)(\kappa(p-2)-sp)- (p-1-sp).
\]
We start by checking that $f(1)>0$. In fact, in this case,
\[
    \theta(1)=\frac{p-1-sp+\varepsilon_1}{p(1-s)},
\]
and so we have
\begin{align*}
    f(1)=&\,1+(p-1-sp+\varepsilon_1)\frac{p-2-sp}{p(1-s)}-(p-1-sp)\\
    =&\,(p-2-sp)\frac{-1+\varepsilon_1}{p(1-s)}>0.
\end{align*}
Therefore, by continuity of $f$, there exists a small enough $\varepsilon_2>0$ and a large enough $\kappa < 1$ so that \eqref{eq:kappa} holds. We assume this condition in addition to \eqref{eq:cond_kappa}.  Combining the estimates above, we obtain 
\[
    \II_3\geq -C_1 |a_K|^{p(1-s)-1+\varepsilon},
\]
for $\varepsilon=\min\{\varepsilon_1,\varepsilon_2\}$. By Lemma \ref{l:estimate_I4},
\[
    \II_4\geq -C_2|a_K|^\kappa,
\]
and so, by Lemma \ref{l:estimate_time},
\[
    \II_1+\II_2+\II_3+\II_4\leq C_3|a_K|^\frac{\kappa \beta^\ast}{1+\beta^\ast}.
\]
Combining all the estimates, we get
\[
    c_1L^{p-1}\delta_0^{d+1}|a_K|^{p(1-s)-1}\leq C_1 |a_K|^{p(1-s)-1+\varepsilon}+C_2|a_K|^\kappa+C_3|a_K|^\frac{\kappa \beta^\ast}{1+\beta^\ast}.
\]
As in the final part of Step 4, we choose $\beta^\ast$ large enough so that
\[
    \kappa \frac{\beta^\ast}{1+\beta^\ast} > p(1-s) - 1,
\]
which is possible since $\kappa > p(1-s) - 1$. As before, this leads to a contradiction. The proof is complete.
\end{proof}

\section{Time regularity and proof of Theorem \ref{t:Lipschitz}}\label{s:time-variable-estimates}

In this section, we prove the regularity result with respect to time and conclude with the proof of the main theorem. The proof of the time regularity uses the spatial regularity to construct a suitable barrier. This is, by now, a standard argument in the local case, see \cite[Lemma 9.1]{BBO} or, more recently, \cite[Lemma 3.1]{IJS}. We start from the assumption that $u$ is Lipschitz in space, as proven in the previous section.

\begin{proposition}\label{p:time-regularity-estimate-phip-desc}
Assume $u$ is a viscosity solution to \eqref{eq:parabolic-fractional-p-laplacian} in $Q_1$ satisfying \eqref{assumption:normalized-setting}. In addition, suppose that \eqref{assumption:bat-ind} holds with $\kappa = 1$. Then, there exists a constant $C>0$ such that
\[
    \sup_{x \in B_{1/2}}|u(x,t_1) - u(x,t_2)| \leq C |t_1-t_2|^{\alpha},
\]
for every $t_1,t_2 \in (-1/4,0]$.  The exponent $\alpha$ is given by
\begin{align*}
    \alpha = \min\left\{ \left( \dfrac{1}{1 - q_c}\right)^-,1 \right\},
\end{align*}
where $q_c \coloneqq - 1+p(1-s)$.

In particular, $u$ is Lipschitz in time when $q_c>0$.
\end{proposition}

\begin{proof}
Fix $t_0 \in (-1/4,0)$, consider the cylinder $B_1 \times (t_0,0]$, and define
\[
    \Psi_\eta(x,t) \coloneqq \eta + L_0(t-t_0) + L_1|x|^\gamma,
\]
for $\eta>0$ and universal parameters $L_0$ and $L_1$ that depend further on $\eta$, and
\[
    \gamma = \max\left\{\left(1 - \frac{q_c}{p-1} \right)^+,1\right\}.
\]
This choice ensures that the local part of $(-\Delta_p)^s\Psi_\eta$ is finite. To streamline the discussion, we delay this computation to the end of the proof.

Let us compare the values of the difference $u(x,t) - u(0,t_0)$ with $\Psi_\eta$ on the region $(\overline{B_1} \setminus B_{1/2}) \times (t_0,0]$ and $B_1 \times \{t_0\}$. For $x \in \overline{B_1} \backslash B_{1/2}$, we have
\[
    u(x,t) - u(0,t_0) \leq 2\|u\|_{L^\infty(Q_1)}\leq 2,
\]
and so for $L_1 \geq 2^{1+\gamma}$, we have
\[
    u(x,t) - u(0,t_0) \leq \Psi_\eta(x,t).
\]
For $x \in B_1$ and $t=t_0$, we prove that by making $L_1$ large enough, we have
\[
    u(x,t_0) - u(0,t_0) \leq \Psi_\eta(x,t_0) = \eta + L_1|x|^\gamma.
\]
Indeed, since $u$ is Lipschitz in space, we know that there is a constant $C>0$ such that
\[
    |u(x,t_0) - u(0,t_0)| \leq C|x|, \quad \text{for }  x \in B_{1/2},
\]
and thus we only need to pick $L_1$ such that
\[
    C|x| \leq \eta + L_1|x|^\gamma, \quad \text{for } x \in B_{1/2}.
\]
If we choose
\[
    L_1 \geq C\left(\frac{\eta}{C} \right)^{1-\gamma},
\]
then we have, for $|x| > \eta/C$, that
\[
    C|x| = C |x|^{1 - \gamma}|x|^\gamma < C \left(\frac{\eta}{C}\right)^{1-\gamma}|x|^\gamma.
\]
On the other hand, for $|x|\leq \eta/C$, we directly have
\[
   C|x| \leq \eta \leq \eta + L_1|x|^\gamma. 
\]
Therefore, the constant $L_1$ has to satisfy
\[
    L_1 \geq \max\left\{2^{1+\gamma}, C\left(\frac{\eta}{C} \right)^{1-\gamma} \right\},
\]
and so we choose $L_1 = C\eta^{1-\gamma}$, for $C$ large enough. Note that the above computations simplify considerably when $\gamma=1$, but the same conclusions still hold. We have shown that
\begin{align}\label{eq:compa_psi}
    u- u(0,t_0) \leq \Psi_\eta \quad \text{on} \quad \left(\overline{B}_1 \backslash B_{1/2}\right) \times (t_0,0] \quad \text{and} \quad B_1 \times \{ t_0\}. 
\end{align}
Defining 
\begin{align*}
    v(x,t) \coloneqq \begin{cases}
     \Psi_\eta(x,t) & \text{ for } (x,t) \in B_1 \times (t_0,0],\\
    u(x,t)- u(0,t_0) & \text{ otherwise},
\end{cases}
\end{align*}
it follows from \eqref{eq:compa_psi} that
\[
    u- u(0,t_0) \leq v  \quad \text{on} \quad \left(\R^d \backslash B_{1/2}\right) \times (t_0,0] \quad \text{and} \quad B_1 \times \{t_0\}.
\]
We wish to use the comparison principle to prove that this inequality propagates to the interior cylinder $B_{1/2} \times (t_0,0]$. We start with the following claim, which we justify at the end of the proof.

\textbf{Claim:} There exists a large constant $C$ such that, if $L_0 = CL_1^{p-1}$, then
\[
    \partial_t \Psi_\eta + (-\Delta_p)^s v > 0 \quad \text{in} \quad B_{1/2} \times (t_0,0].
\]
If the claim holds, then we would get, by the comparison principle, that
\[
    u(x,t) - u(0,t_0) \leq \Psi_\eta(x,t) \quad \text{on} \quad B_{1/2} \times (t_0,0].
\]
In particular, we have for any $t \in (t_0,0]$
\[
    u(0,t) - u(0,t_0) \leq \eta + CL_1^{p-1}(t-t_0), 
\]
for any given $\eta>0$. By the choice of $L_1$, we have
\[
    u(0,t) - u(0,t_0) \leq \eta + C\eta^{(1-\gamma)(p-1)}(t-t_0), 
\]
for any $t \in (t_0,0]$ and any $\eta \in (0,1)$. Since, for each fixed $t$, this estimate holds for all $\eta$, we can optimize in the parameter $\eta$ in the following way. If $\gamma=1$, then we can just take $\eta=0$. Otherwise, call
\[
    f(\eta) = \eta+c\eta^{(1-\gamma)(p-1)},
\]
with $c = C(t-t_0)$. We see that $f'(\eta)=0$ when
\[
    \eta = \left(c(\gamma-1)(p-1)\right)^\frac{1}{(\gamma-1)(p-1)+1},
\]
and so, we reach, as intended,
\[
    u(0,t) - u(0,t_0) \leq C(t-t_0)^{\frac{1}{1+(p-1)(\gamma-1)}}.
\]

To prove the claim, we calculate explicitly the value of $(-\Delta_p)^s v(\bar x,\bar t)$ for any $(\bar x, \bar t) \in B_{1/2} \times (t_0,0]$, which we split into the local and nonlocal parts
\[
    \LL_1 \coloneqq \int_{B_1}\frac{J_p(v(\bar{x},\bar t) - v(x,\bar t))}{|x-\bar{x}|^{d + sp}}\dd x, \quad  \LL_2 \coloneqq \int_{\R^d \backslash B_1}\frac{J_p(v(\bar{x},\bar t) - v(x,\bar t))}{|x-\bar{x}|^{d + sp}}\dd x.
\]
For the quantity $\LL_1$, we have
\begin{align*}
    \LL_1  = \int_{B_1}\frac{J_p(\Psi_\eta(\bar{x},\bar{t}) - \Psi_\eta(x,\bar{t}))}{|x-\bar{x}|^{d + sp}}\dd x = L_1^{p-1} \int_{B_1}\frac{J_p(\phi(\bar{x}) - \phi(x))}{|x-\bar{x}|^{d + sp}}\dd x,
\end{align*}
where $\phi(x) \coloneqq |x|^\gamma$. 

When $\gamma = 1$, corresponding to the case $q_c>0$, we simply use that
\[
    |\phi(x) - \phi(\bar x)|^{p-1} \leq |x- \bar x|^{p-1},
\]
and then it readily follows that $|\LL_1|\leq C$. In the case $q_c \leq 0$, we proceed as follows: since $\gamma > 1$, by the $C^{1,\gamma-1}$ regularity of $\phi$, we have
\[
    |\phi(x) - \phi(\bar x)| \leq \gamma|\bar x|^{\gamma-1}|x-\bar x| + C |x-\bar x|^\gamma,
\]
where $C \coloneqq \|\phi\|_{C^{1,\gamma-1}(B_3)}$. If $\bar x = 0$, we can compute the integral directly, and so we assume otherwise. In this case,
\begin{equation}\label{phi_c1gamma}
    |\phi(x) - \phi(\bar x)|^{p-1} \leq C\big(|\bar x|^{(p-1)(\gamma-1)} |x-\bar x|^{p-1} + |x-\bar x|^{(p-1)\gamma}\big), 
\end{equation}
for $x \in B_2(\bar x) \subset B_3$. For $r = |\bar x|/2$, we split
\begin{align*}
    L_1^{-(p-1)} \LL_1 & = \mathrm{P.V.} \int_{B_r(\bar x)}\frac{J_p(\phi(\bar{x}) - \phi(x))}{|x-\bar{x}|^{d + sp}}\dd x + \int_{B_1 \backslash B_r(\bar x)}\frac{J_p(\phi(\bar{x}) - \phi(x))}{|x-\bar{x}|^{d + sp}}\dd x\\
    & = M_1 + M_2.
\end{align*}
In the ball $B_r(\bar x) \subset B_1$, the function $\phi$ is smooth. By Taylor's expansion, we have
\[
    \phi(\bar x+h) = \phi(\bar x) + D\phi(\bar x) \cdot h + R(|h|), \quad \text{where} \quad R(h) \approx r^{\gamma-2}|h|^2. 
\]
Then, by symmetry of the domain and the kernel, we have
\begin{align*}
    M_1 & = \mathrm{P.V.} \int_{B_r}\frac{J_p(-D\phi(\bar x) \cdot h - R(|h|))}{|h|^{d + sp}}\dd h\\
        & = \mathrm{P.V.} \int_{B_r}\frac{J_p(-D\phi(\bar x) \cdot h - R(|h|)) - J_p(-D\phi(\bar x) \cdot h)}{|h|^{d + sp}}\dd h.
\end{align*}
By Lemma \ref{l:Jp_est}, we obtain
\begin{align*}
   |M_1| & \leq C\int_{B_r}\frac{(r^{\gamma-1}|h| + r^{\gamma-2}|h|^2)^{p-2}r^{\gamma-2}|h|^2}{|h|^{d + sp}}\dd h\\
    & \leq C\int_{B_r}\frac{r^{(\gamma-1)(p-2) + \gamma-2}|h|^{p}}{|h|^{d + sp}}\dd h\\
    & = Cr^{\gamma(p-1) - sp}.
\end{align*}
Thus, since $\gamma(p-1) - sp>0$, we obtain that $|M_1| \leq C$. To estimate the term $M_2$ we use \eqref{phi_c1gamma} to obtain
\begin{align*}
    |M_2| & \leq  \int_{B_1 \setminus B_{r}(\bar x)} |\bar x|^{(\gamma-1)(p-1)} |x-\bar x|^{p-1-d-sp}\dd x\\
    & \quad + \int_{B_1} |x-\bar x|^{(p-1)\gamma-d-sp}\dd x \bigg)\\
    & \eqqcolon I_1 + I_2.
\end{align*}
Regarding $I_1$, 
\begin{align*}
    I_1&\leq  C\int_{B_2(\bar x)\setminus B_r(\bar x)}|\bar x|^{(\gamma-1)(p-1)} |x-\bar x|^{p-1-d-sp}\dd x\\
    &\leq Cr^{(\gamma-1)(p-1)}\int_r^2 \tau^{p-2-sp}\dd \tau\\
    &\leq C r^{(\gamma-1)(p-1)+p-1-sp}\\
    &\leq C.
\end{align*}
Finally, the term $I_2$ is bounded directly using polar coordinates
\begin{align*}
    I_2&\leq\int_{B_2(\bar x)} |x-\bar x|^{(p-1)\gamma-d-sp}\dd x\\
    &=C\int_0^2 \tau^{(p-1)\gamma-sp-1}\dd \tau\\
    &\leq C.
\end{align*}

Therefore we obtain
\[
    |\LL_1| \leq C L_1^{p-1}.
\]

For the quantity $\LL_2$, we use the monotonicity of $J_p$ to get
\begin{align*}
    \LL_2 & = \int_{B_1^c} \frac{J_p\bigl(\eta + L_0(\bar t - t_0) + L_1|\bar x|^\gamma -u(x,\bar t) + u(0,t_0)\bigr)}{|x-\bar x|^{d+sp}}\dd x\\
    & \geq \int_{B_1^c} \frac{J_p\bigl(u(0,t_0) -u(x,\bar t)\bigr)}{|x-\bar x|^{d+sp}}\dd x \geq -C, 
\end{align*}
where we used Assumption \eqref{assumption:normalized-setting}. We combine these estimates to obtain
\[
    \partial_t \Psi_\eta + (-\Delta_p)^s v \geq L_0 - CL_1^{p-1} - C > 0
\]
for $L_0=C L_1^{p-1}$ where $C$ is large, universal.

This concludes the claim and thus the proof.
\end{proof}

We now present the proof of Theorem \ref{t:Lipschitz}.

\begin{proof}[Proof of Theorem \ref{t:Lipschitz}]
Fix $(x_0,t_0)\in Q_{1/2}$ and set
\[
    \mathcal{M} \coloneqq \|u\|_{L^\infty(Q_{3/4})} + \sup_{t \in (-(3/4)^2,0]}\|u(\cdot,t)\|_{L^{p-1}_{sp}(\R^d)}    .
\]
Note that from \cite{S}, we know that $\|u\|_{L^\infty(Q_{3/4})}$ is finite.

For $\lambda>0$, define the rescaled function
\[
    v(x,t)\coloneqq \mathcal{M}^{-1}u(\lambda x+x_0,\,\mathcal{M}^{2-p}\lambda^{sp}t+t_0).
\]
For $\lambda$ sufficiently small, $v$ solves the same equation in the weak sense in $Q_4$ and satisfies assumption \eqref{assumption:normalized-setting}. By Theorem \ref{t:NL_Bra_nobo}, it also satisfies assumption \eqref{assumption:bat-ind} for $\kappa=\kappa_0$.
Moreover, by Theorem \ref{t:weak-are-viscosity}, $v$ is also a viscosity solution of \eqref{eq:parabolic-fractional-p-laplacian} in $Q_1$. Hence, we are in a position to apply Theorem \ref{t:lip-in-space_Chibrata}, yielding
\[
    |v(x,0)-v(y,0)|\le C|x-y|, \quad \text{for } x,y\in B_{1/2}.
\]
We now invoke Proposition \ref{p:time-regularity-estimate-phip-desc} to obtain
\[
    \sup_{x\in B_{1/4}}|v(x,t)-v(x,0)|\le C|t|^\alpha, \quad \text{for } t\in(-1/4,0],
\]
where $\alpha$ is as in Proposition \ref{p:time-regularity-estimate-phip-desc}. Combining the previous estimates, we conclude that
\[
    |v(x,t)-v(0,0)|\le C(|x|+|t|^\alpha), \quad \text{for } x\in B_{1/2}\text{ and } t\in(-1/4,0].
\]

Rescaling back to $u$ gives
\[
    |u(x,t)-u(x_0,t_0)|\le C\mathcal{M}\bigl(|x-x_0|+\mathcal{M}^{(p-2)\alpha}|t-t_0|^\alpha\bigr),
\]
for all
\[
    (x,t)\in B_\lambda(x_0)\times\bigl(t_0-\mathcal{M}^{2-p}\lambda^{sp},\,t_0\bigr].
\]
This proves the desired estimate in a (small) cylinder centered at $(x_0,t_0)$. By a covering argument, we obtain
\begin{align*}
    |u(x,t) - u(y,\tau)| & \leq C\mathcal{M}\left(|x-y|+\mathcal{M}^{(p-2)\alpha}|t-\tau|^\alpha \right),
\end{align*}
for any $(x,t), (y,\tau) \in Q_{1/2}$.
\end{proof}

\medskip

{\small \noindent{\bf Acknowledgments.} The authors thank Davide Giovagnoli and Luis Silvestre for insightful conversations. This publication is based upon work supported by King Abdullah University of Science and Technology (KAUST) under Award No. ORFS-CRG12-2024-6430.}

\medskip

\bibliographystyle{amsalpha}

\end{document}